\documentclass{article}
\usepackage[english]{babel}

\usepackage[letterpaper,top=2cm,bottom=2cm,left=2cm,right=2cm,marginparwidth=1.75cm]{geometry}

\usepackage{calrsfs}
\usepackage{amssymb}
\usepackage{amsmath}
\usepackage{amsthm}
\usepackage{subcaption}
\usepackage{multirow}
\usepackage{graphicx}
\usepackage{amsfonts}
\usepackage{mathtools}
\usepackage{mathrsfs}
\usepackage{placeins}

\newtheorem{remark}{Remark}

\newtheorem{assumption}{Assumption}
\newtheorem{theorem}{Theorem}
\newtheorem{proposition}{Proposition}

\usepackage[mathscr]{euscript}
\usepackage[dvipsnames]{xcolor}
\usepackage[colorlinks=true, allcolors=OliveGreen]{hyperref}
\usepackage{filecontents}
\usepackage{tikz}
\usepackage{pgfplots}
\usetikzlibrary{external}
\usepackage{authblk}
\usepackage{todonotes}
\usepackage{array}
\newcolumntype{C}{>{\centering\arraybackslash}p{2.5cm}}

\newcommand{\averagel}{\{\!\!\{}
\newcommand{\averager}{\}\!\!\}}

\newcommand{\jumpl}{[\![}
\newcommand{\jumpr}{]\!]}
\newcommand{\tnorm}{|\!\!\:|\!\!\:|}

\newcommand{\partition}{\mathcal{T}_h}
\newcommand{\facesinternal}{\mathcal{F}^\mathrm{I}_h}
\newcommand{\faces}{\mathcal{F}_h}
\newcommand{\facesN}{\mathcal{F}_h^\mathrm{N}}
\newcommand{\facesD}{\mathcal{F}_h^\mathrm{D}}
\newcommand{\facesboundary}{\mathcal{F}^\mathrm{B}_h}

\newcommand{\Wh}{W_{h}^\mathrm{DG}}
\DeclareMathAlphabet{\mathcalligra}{T1}{calligra}{m}{n}

\mathtoolsset{showonlyrefs,showmanualtags}

\graphicspath{{Images/}}

\title{Discontinuous Galerkin approximations of the heterodimer model for protein-protein interaction\footnote{\textbf{Funding}: This research has been partially funded by the European Union (ERC, NEMESIS, project number 101115663). Views and opinions expressed are however those of the author(s) only and do not necessarily reflect those of the European Union or the European Research Council Executive Agency. PFA has been partially funded by PRIN2020 n. 20204LN5N5“Advanced polyhedral discretisations of heterogeneous PDEs for multiphysics problems” research grant, funded by the Italian Ministry of Universities and Research (MUR).  FB is partially funded by “INdAM - GNCS Project”, codice CUP E53C22001930001. The present research is part of the activities of “Dipartimento di Eccelllenza 2023-2027”. MC, FB, and PFA are members of INdAM-GNCS.}}

\author[1]{Paola F. Antonietti\footnote{paola.antonietti@polimi.it}}

\author[1]{Francesca Bonizzoni\footnote{francesca.bonizzoni@polimi.it}}

\author[1]{Mattia Corti\footnote{mattia.corti@polimi.it}}

\author[1]{Agnese Dall'Olio\footnote{agnese.dallolio@mail.polimi.it}}

\affil[1]{MOX-Dipartimento di Matematica, Politecnico di Milano, Piazza Leonardo da Vinci 32, Milan, 20133, Italy}

\begin{document}
\maketitle

\begin{abstract}
Mathematical models of protein-protein dynamics, such as the heterodimer model, play a crucial role in understanding many physical phenomena, e.g., the progression of some neurodegenerative diseases. This model is a system of two semilinear parabolic partial differential equations describing the evolution and mutual interaction of biological species. This article presents and analyzes a high-order discretization method for the numerical approximation of the heterodimer model capable of handling complex geometries. In particular, the proposed numerical scheme couples a Discontinuous Galerkin method on polygonal/polyhedral grids for space discretization, with a $\theta$-method for time integration. This work presents novelties and progress with respect to the mathematical literature, as stability and a-priori error analysis for the heterodimer model are carried out for the first time. Several numerical tests are performed, which demonstrate the theoretical convergence rates, and show good performances of the method in approximating traveling wave solutions as well as its flexibility in handling complex geometries. Finally, the proposed scheme is tested in a practical test case stemming from neuroscience applications, namely the simulation of the spread of $\alpha$-synuclein in a realistic test case of Parkinson’s disease in a two-dimensional sagittal brain section geometry reconstructed from medical images.

\end{abstract} 
\section{Introduction}
\label{sec:introduction}
The heterodimer continuum model \cite{matthaus_diffusion_2006} is a well-recognized mathematical model in the context of the description of protein-protein interactions \cite{weickenmeier_multiphysics_2018, fornari_spatially-extended_2020}. From a mathematical point of view, this model is a system of parabolic semilinear partial differential equations (PDEs) \cite{smoller_shock_1994}. Its formulation allows the description of the underlying microscopic mechanism, typical of protein dynamics, while simultaneously capturing the broader macroscopic phenomena of propagation in spatial domains. 
\par
An important application of this model concerns the study of prionic protein dynamics in neurodegenerative diseases. Indeed, in those pathologies known as proteinopathies (i.e. Alzheimer's and Parkinson's diseases), the heterodimer model embraces a complete consideration of both healthy and misfolded protein configurations, typically related to the pathology development \cite{juckerSelfpropagationPathogenicProtein2013}. Furthermore, the model includes explicit terms representing distinct molecular kinetic processes: production, clearance \cite{wilsonHallmarksNeurodegenerativeDiseases2023}, and conversion, which is typically guided by the presence of other misfolded proteins \cite{hasegawaPrionlikeMechanismsPotential2017}. This level of detail of the mathematical model affords the ability to simulate different scenarios \cite{fornariPrionlikeSpreadingAlzheimer2019}, such as variations in clearance or production mechanisms. The development of mathematical models and computational methods in this field, where the neurodegeneration may span over decades and the clinical data are not easy to collect,
represents an essential tool to support medical science and a key instrument to asses patient-specific evolution scenarios. In biological applications, and in particular, in the field of neurodegenerative diseases, a well-known model is the Fisher-Kolmogorov (FK) equation (see \cite{weickenmeierPhysicsbasedModelExplains2019,fornari_spatially-extended_2020, corti_discontinuous_2023}). The latter can be interpreted as a particular case of the heterodimer model, when the dynamic of one population is constant (i.e. healthy protein concentration \cite{weickenmeierPhysicsbasedModelExplains2019}). 
Instead, the heterodimer model accommodates the dynamics of both biological populations, hence delivering a more informative solution. The price to pay is a substantially increased complexity of the mathematical formulation, which involves not just one equation, but rather a system of two equations, together with a larger computational effort required to approximate the solution of the model.

\par
Concerning the numerical approximation of the heterodimer model, in literature we find many different methods, i.e. finite element methods \cite{fornariPrionlikeSpreadingAlzheimer2019} or reduced order network diffusion models \cite{thompsonProteinproteinInteractionsNeurodegenerative2020a}. For what concerns numerical methods for semilinear parabolic problems, there exist some works concerning Discontinuous Galerkin approximation for the vectorial case \cite{cangiani_discontinuous_2013,cangiani_discontinuous_2016,cangiani_spatial_2010}, as well as for the scalar case \cite{lasis_hp-version_2007,corti_discontinuous_2023,corti_positivity_preserving,bonizzoni_structure-preserving_2020}.  In this work, we propose and analyze a Discontinuous Galerkin formulation on polygonal/polyhedral grids (PolyDG) for the space discretization of the heterodimer system, coupled with a Crank-Nicolson time discretization scheme, with a semi-implicit treatment of the nonlinear terms. 
\par
The proposed approach enjoys distinguishing features and relevant advantages, when compared to the available approaches in the literature. First, it supports high-order approximations that are able to correctly reproduce traveling-wave solutions. Second, it is capable to handle complex geometries. This is due to the possibility of using arbitrarily shaped polygonal/polyhedral elements, possibly obtained by means of agglomeration techniques \cite{antonietti_agglomeration_2024}. Finally, PolyDG methods allow the local tuning of discretization parameters, like the polynomial order \cite{cangianiHpVersionDiscontinuousGalerkin2017}. This flexibility becomes particularly advantageous in the accurate approximation of propagating wavefronts, even within meshes that are not excessively refined.
\par
The main contributions of the present work are the following. On the one hand, we study the PolyDG semi-discrete formulation. The performed analysis contains an important improvement with respect the previous results for the FK equation in \cite{corti_discontinuous_2023}: indeed, the exponential dependence on time of the stability bound has been eliminated. Moreover, the analysis is more challenging as it involves quadratic nonlinear coupling terms that must be treated using the Gagliardo-Nirenberg inequality. On the other hand, we prove an \textit{a-priori} error estimate for the PolyDG semi-discrete formulation. Also in this case, some improvements have been achieved when compared to \cite{corti_discontinuous_2023}, namely, error estimates are devied without a assumptions on the model parameters.
To the best of our knowledge, these are the first results of this type for the semi-discrete form of the heterodimer model. 
\par
We consider as the application of the work, the simulation of $\alpha$-synuclein spreading in Parkinson's disease, whose polymerization leads to the formation of Lewy bodies \cite{spillantini_-synuclein_1997,hasegawaPrionlikeMechanismsPotential2017}. Sophisticated mathematical representation would provide valuable insights into the disease progression \cite{masel_quantifying_1999,carbonell_mathematical_2018} since the protein detection is only possible by histological examination of postmortem brain tissue \cite{korat_alpha-synuclein_2021}. Indeed, chemical ligands for positron emission tomography imaging (PET) do not exist \cite{korat_alpha-synuclein_2021}. 

\bigskip
The paper is organized as follows. Section~\ref{MathematicalModel} presents the heterodimer model. Section~\ref{PolyDGFormulation} is dedicated to deriving the semidiscrete PolyDG formulation. In Section~\ref{Analysis} we prove stability and \textit{a-priori} error estimates. In Section~\ref{TimeDiscret} we present the fully discrete formulation of the problem. In Section~\ref{Verification} theoretical convergence rates are verified via numerical simulations with known exact solutions. We also demonstrate the ability of the model to reproduce a traveling-wave solution. In Section~\ref{BrainApp}, we present an application in the context of neurodegenerative disorders, namely the spreading of $\alpha$-synuclein protein within a two-dimensional brain section. Finally, in Section~\ref{conclusion}, we draw some conclusions and discuss future developments.
\section{The mathematical model}
\label{MathematicalModel}
The equations of the heterodimer model that govern the description of the spatial and temporal dynamics of two biological populations $c(\boldsymbol{x},t)$  $q(\boldsymbol{x},t)$ within an open bounded domain $\Omega \subset \mathbb{R}^{d}$ ($d=2,3$) with Lipschitz boundary, over a time interval $(0,T]$ is the following:
\begin{equation}
\label{eq:HM_governingequation}
\begin{cases}
    \dfrac{\mathrm{d}c}{\mathrm{d}t}  =\nabla \cdot (\textbf{D} \nabla c) - k_{1} \, c - k_{12} \, c\, q + f_c &  \mathrm{in} \: \Omega \times (0,T], \\[6pt]
    \dfrac{\mathrm{d}q}{\mathrm{d}t}  = \nabla \cdot (\textbf{D} \nabla q) - \tilde{k}_1\, q + k_{12}\, q \, c + f_q & \mathrm{in} \: \Omega \times (0,T], \\[6pt]
    c = c_\mathrm{D} \mathrm{,} \quad q  = q_\mathrm{D} & \mathrm{on}  \; \Gamma_\mathrm{D}\times (0,T], \\[6pt]
    (\textbf{D}\nabla c )\cdot \boldsymbol{n} = 0 \mathrm{,} \quad (\textbf{D} \nabla q )\cdot \boldsymbol{n}  = 0  & \mathrm{on} \; \Gamma_\mathrm{N}\times (0,T], \\[6pt]
    c(0,\boldsymbol{x}) = c_{0}\mathrm{,} \quad q(0,\boldsymbol{x})  = q_{0} & \mathrm{in} \: \Omega.
\end{cases}
\end{equation}
The model takes the form of a time-dependent system of PDEs coupled through a non-linear reaction term. Both populations are affected by a biological destruction process at rate $k_1=k_1(\boldsymbol{x})$ and $\tilde{k}_1=\tilde{k}_1(\boldsymbol{x})$, respectively. The agents of first population are converted into ones of the second at a rate $k_{12} = k_{12}(\boldsymbol{x})$. The forcing terms $f_c=f_c(\boldsymbol{x},t)$ and $f_q=f_q(\boldsymbol{x},t)$ correspond to a generic external introduction of mass into the system. Suitable initial and boundary conditions complete the model. The boundary of the domain $\partial \Omega$ is partitioned into two disjoint subsets, namely $\Gamma_\mathrm{D}$ and $\Gamma_\mathrm{N}$. On $\Gamma_\mathrm{D}$, we impose fixed concentration values, while on $\Gamma_\mathrm{N}$, we enforce zero flux conditions, with $\boldsymbol{n}$ denoting the normal vector pointing outward from the domain. The initial conditions $q_0(\boldsymbol{x})$ and $c_0(\boldsymbol{x})$ correspond to initial states of the populations.
\par
\begin{assumption}
\label{CoefficientsRegularity}
For the subsequent analysis of the problem we require the following regularity for the data, forcing terms and boundary conditions:
\begin{itemize}
    \item  $\mathbf{D}\in L^{\infty}(\Omega,\mathbb{R}^{d\times d})$ and $\exists d_0 > 0: \; d_0\lvert\boldsymbol{\xi}\rvert^2\le\boldsymbol{\xi}^{\top}\mathbf{D}\boldsymbol{\xi} \quad \forall \boldsymbol{\xi}\in \mathbb{R}^d$;
    \item $k_{1},\tilde{k}_{1},k_{12}\in L^{\infty}_+(\Omega)$ and $\exists K_1>0:\,K_1\leq k_1$ and $\exists \tilde{K}_1>0:\,\tilde{K}_1\leq \tilde{k}_1$
    \item $f_c(\boldsymbol{x},t)$, $f_q(\boldsymbol{x},t)\in L^2((0,T];L^2(\Omega))$;
    \item $c_\mathrm{D}$, $q_\mathrm{D}\in L^2((0,\mathrm{T}],H^{1/2}(\partial\Omega_\mathrm{D}))$.
\end{itemize}
\end{assumption}
By analyzing the steady-state solutions of~\eqref{eq:HM_governingequation} with constant forcing term in the first equation $f_c(\boldsymbol{x},t)=k_0$ and null in the second $f_q(\boldsymbol{x},t)=0$, it can be proved that the model exhibits two distinct equilibria. The first is an unstable equilibrium at $c =k_0/k_1$ and $q = 0$, representing the extinction of the second population. The second equilibrium $c = \tilde{k}_1/k_{12}$ and $q = k_0/\tilde{k}_1 -k_1/k_{12}$ is stable and corresponds to a situation where both species coexist. This implies that even a small non-zero quantity of the second population, can initiate an irreversible chain reaction of conversion and spread \cite{eigen_prionics_1996}. Furthermore, for the observation of a traveling-wave solution, transitioning the system towards a coexistence state, it is crucial for the coefficients of Equations~\eqref{eq:HM_governingequation} to adhere to a specific relationship \cite{thompsonProteinproteinInteractionsNeurodegenerative2020a}:
\begin{equation}
\label{eq:coefficientsvalue} 
    k_0k_{12}-\tilde{k}_1k_1>0.
\end{equation}
Failure to satisfy this condition would result in having a stable equilibrium point in the first state implying that the model does not exhibits traveling-waves-solutions \cite{smoller_shock_1994, thompsonProteinproteinInteractionsNeurodegenerative2020a}. This is a fundamental difference with respect to the FK model and underlines the importance of parameter tuning for the specific modelling purposes with this type of model. Concerning the well-posedness analysis of the continuous formulation of systems of diffusion-reaction equations and the decay to spatially homogeneous solutions we refer to \cite{smoller_shock_1994}.
\subsection{Heterodimer model applied to neurodegeneration}
The mathematical formulation employed to address the kinetics of prion pathogenesis incorporates the representation of both healthy protein $c$ and misfolded protein $q$ concentrations. Heterodimer model analyzes prion conversion dynamics and the spatial dissemination of these infectious agents within the reference domain \cite{prusinerNovelProteinaceousInfectious1982,prusinerMolecularBiologyPrion1991}. Despite the fact that the prion conversion mechanism remains poorly understood, both the monomeric seeding \cite{prusinerMolecularBiologyPrion1991} and polymeric seeding \cite{harperModelsAmyloidSeeding1997,harrisCellularBiologyPrion1999} theory can be modelled with heterodimer equations. We explain graphically those concepts in Figures \ref{fig:monomeric_seeding} and \ref{fig:polymeric_seeding}, respectively.
\par
\begin{figure}[t]
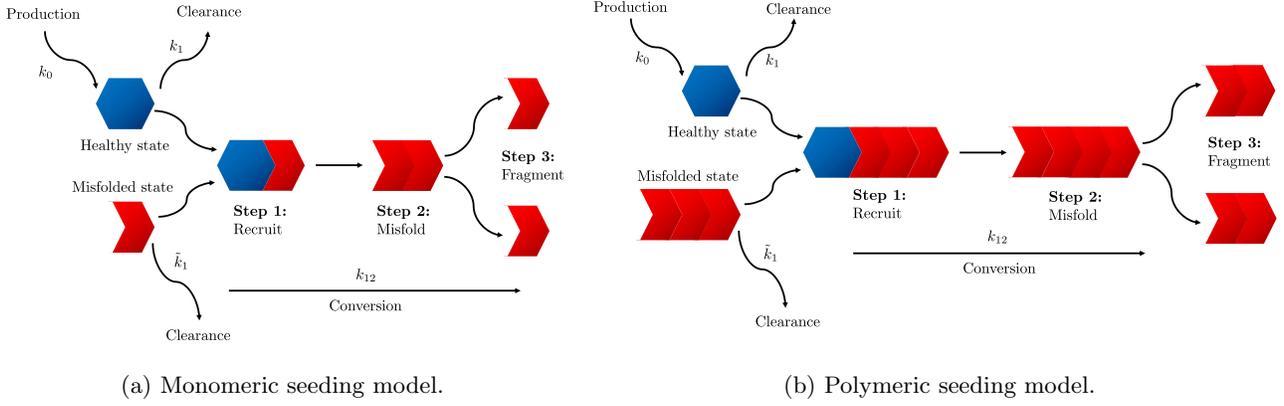

     \begin{subfigure}[b]{0.44\textwidth}
         \centering
         \includegraphics[width=\textwidth]{Monomericheterodimer}
         \caption{Monomeric seeding model.}
         \label{fig:monomeric_seeding}
     \end{subfigure}
          \begin{subfigure}[b]{0.54\textwidth}
         \centering
         \includegraphics[width=\textwidth]{Polymericheterodimer}
         \caption{Polymeric seeding model.}
         \label{fig:polymeric_seeding}
     \end{subfigure}
     \caption{Schematic representation of monomeric and polymeric seeding models.}
    \label{fig:PrionModel}
\end{figure}
\par
Concerning modelling choices in neurological application, we assume that healthy proteins are naturally produced by the organism at a rate $f_c = k_0(\boldsymbol{x})$, while that does not happen for the misfolded form $f_q = 0$. From a macroscopic perspective, the propagation of proteins throughout the domain is captured by the diffusion term in Equation \eqref{eq:HM_governingequation}. The spreading of proteins within the brain occurs via two main mechanisms, extracellular diffusion, and axonal transport, and can be mathematically characterized by employing the diffusion tensor defined as follows \cite{weickenmeierPhysicsbasedModelExplains2019}:
\begin{equation}
    \label{eq:DiffusionTensor}
    \mathbf{D}(\boldsymbol{x}) = d_\mathrm{ext}\mathbf{I} + d_\mathrm{axn}\boldsymbol{\bar{a}}(\boldsymbol{x}) \otimes \boldsymbol{\bar{a}}(\boldsymbol{x}),
\end{equation}
where the second term models the anisotropic diffusion of substances along the axonal directions, denoted by the vector $\boldsymbol{\bar{a}}(\boldsymbol{x})$, which is proportional to $d_\mathrm{axn}>0$. It is commonly assumed that the axonal transport is faster than the isotropic extracellular diffusion, therefore $d_{axn}\geq d_{ext} > 0$ \cite{weickenmeierPhysicsbasedModelExplains2019}. 
\section{Polytopal discontinuous Galerkin semi-discrete formulation}
\label{PolyDGFormulation}
In this section, we aim to present the spatial discretization of Equation~\eqref{eq:HM_governingequation} using the PolyDG method \cite{cangianiHpVersionDiscontinuousGalerkin2017, cangiani_hp-version_2014,antonietti_review_2016,antonietti_bubble_2009}. First, we introduce the symbol $x \lesssim y$ to mean that $\exists C > 0: x \leq C y$, where $C$ may depend on model parameters and the degree of polynomial approximation denoted as $p$, but it is independent of the mesh-size $h$. 
\subsection{Discrete setting: polytopal meshes and functional spaces}
We introduce a partition $\partition$ over the domain $\Omega$, consisting of non-empty, disjoint polygonal or polyhedral elements. Each mesh element $K\in\partition$ is associated with its measure denoted as $\lvert K \rvert$, and its diameter denoted as $h_K$. Consequently, the mesh size parameter is defined as $h =\max_{K\in\partition}\{h_K\}$. Additionally, we consider a partition $\faces$ of the mesh skeleton, which consists of disjoint $(d-1)$-dimensional subsets of $\Omega$ known as mesh faces $F$. These faces are defined in such a way that each face is either an interface between two distinct elements $T_1, T_2 \in \partition$ (where $F\subset T_1 \cap T_2$), or a boundary face (where $F\subset T_1 \cap \partial \Omega$). Mesh interfaces do not necessarily coincide with the faces of polytopic elements in the partition $\partition$, allowing for the handling of non-conforming partitions that involve hanging nodes, and partitions where the faces may not be planar or even connected \cite{dipietro:HHO}.
\par
In the two-dimensional setting, mesh faces correspond to line segments, while in three-dimensional domains, mesh faces consist of general polygonal surfaces that are further assumed to be subdivided into triangles. Henceforth, $\faces$ will denote the collective set of mesh faces or all open triangles belonging to a sub-triangulation of the mesh faces in three-dimensional domains. Consequently, $\faces$ is consistently defined as a set of $(d-1)$-dimensional simplexes. The set $\faces$ can be subdivided into two distinct subsets: $\facesinternal$, which comprises all interior interfaces, and $\facesboundary$, which consists of all boundary faces. The boundary faces can be further partitioned based on the type of boundary conditions specified in Equation~\eqref{eq:HM_governingequation}. Specifically, we denote by $\facesD$ as the portion of the boundary where Dirichlet conditions are imposed and by $\facesN$ as the portion where Neumann conditions are set. Additionally, we assume that the triangulation $\partition$ is constructed in such a way that any face $F\in\faces$ is entirely contained within either the Dirichlet boundary region, denoted as $\Gamma_\mathrm{D}$, or the Neumann boundary region, denoted as $\Gamma_\mathrm{N}$.
\par
We make the assumption that the partition $\partition$ satisfies the shape regularity conditions outlined in \cite{cangiani_hp-version_2014,antonietti_review_2016}. These conditions ensure the fulfilment of crucial inequalities such as the trace inequality, the inverse inequality, and the approximation property of the $L^2$-projection on general polygonal/polyhedral meshes. Furthermore, we impose additional regularity assumptions as detailed in  \cite{buffa_compact_2009,nirenbergdiscrete} to extend the Gagliardo-Nirenberg inequality into a discrete setting. 
\par
\begin{assumption}
\label{Mesh_quality_assumption}
  We assume that the sequences of meshes $\{\partition\}_h$ satisfies the following conditions:
  \begin{itemize}
      \item (Shape regularity) for every $K\in\partition$ and each $h\in(0,1]$ it holds:
      \begin{equation*}
           h_{K}^d \lesssim \lvert K \rvert \lesssim h_K^d;
      \end{equation*}
      \item (Contact regularity) for every $F\in\faces$ with $F \subset K$ for some $K\in\partition$ and $h\in(0,1]$ it holds $h_K^{d-1} \lesssim \lvert F \rvert$, where $\lvert F \rvert$ denotes the Hausdorff measure of $F$;
      \item (Submesh condition) For each $h\in(0,1]$, there exists a shape-regular, conforming, matching simplicial submesh \cite{dipietro:HHO} $\widetilde{\partition}$ such that
      \begin{enumerate}
          \item For each $\widetilde{K}\in \widetilde{\partition}$ there exists $K\in \partition$ such that $\widetilde{K} \subseteq K$,
          \item The family $\{\widetilde{\partition}\}_h$ satisfies shape and contact regularity assumptions,
          \item for any $\widetilde{K}\in\widetilde{\partition}$ and any $K\in\partition$ with $\widetilde{K}\subseteq K$, it holds that $h_K \lesssim h_{\widetilde{K}}$.
      \end{enumerate}
  \end{itemize}
\end{assumption}
We proceed by defining the functional spaces relevant to our analysis. 
We define the Broken Sobolev space:
\begin{equation}
\label{eq:brokensob}
    W_q = H^q(\Omega, \partition) = \{ v \in L^2(\Omega): \quad v\vert_{K} \in H^q(K) \; \forall K \in \partition\}.
\end{equation}
We will seek discrete solutions to the problem within the following discontinuous finite element space:
\begin{equation*}
    \Wh = \{ w\in L^2(\Omega): \quad w\vert_K \in \mathbb{P}_p(K) \; \forall K\in \partition \},
\end{equation*}
where $\mathbb{P}_p(K)$ is the space of polynomials of total degree $p \geq 1$ defined over the mesh element $K$.
\par
In conclusion, we recall some notation consistently followed throughout the remainder of the exposition. For vector-valued and tensor-valued functions previous definitions are to be intended as extended componentwise. From now on, all integrals over the domain $\Omega$ should be intended as the sum of the contributions over each element $K$ of the partition $\partition$, in particular, $\| \cdot \| = (\sum_{K\in \partition} \| \cdot \|_{L^2(K)}^2)^{1/2}$.  Analogously, the $L^2$-norm on a set of faces $\faces$ will be indicated as $\|\cdot \|_{\faces} = (\sum_{F\in\faces}\| \cdot \|_{L^2(F)}^2)^{1/2}$.
\par
We introduce the following trace-related operators \cite{arnold_unified_2002}. Each interface $F \in \facesinternal$ is shared by two mesh elements denoted by $K^{\pm}$, we denote by $\boldsymbol{n}^{\pm}$ the unit normal vector on face $F$ directed outward with respect to $K^{\pm}$, respectively. For any sufficiently regular scalar-valued function $v$ and vector-valued function $\boldsymbol{w}$ we set:
\begin{itemize}
    \item the jump operator $\jumpl \cdot \jumpr$ on $F \in \facesinternal$ as $\jumpl v \jumpr = v^+ \boldsymbol{n}^+ + v^-\boldsymbol{n}^- $ and $\jumpl \boldsymbol{w}  \jumpr = \boldsymbol{w}^+ \cdot \boldsymbol{n}^+ + \boldsymbol{w}^- \cdot \boldsymbol{n}^-$,
    \item the average operator $\averagel \cdot \averager$ on $F \in \facesinternal$ as $\averagel v \averager = \frac{1}{2}(v^+ + v^-)$  and $\averagel \textbf{w} \averager = \frac{1}{2}(\textbf{w}^+ + \textbf{w}^-)$,
\end{itemize}
where $v^{\pm}$ and  $\boldsymbol{w}^{\pm}$ denote the traces on $F$ of the function defined over $K^{\pm}$, respectively. Analogously, we define the jump and average operators on the faces $F \in \facesD$, associate with the boundary mesh element $K$ with outward unit normal vector $\boldsymbol{n}$ as: 
\begin{itemize}
    \item the jump operator $\jumpl \cdot \jumpr$ on $F \in \facesD$ with Dirichlet condition $g$ and $\boldsymbol{g}$: $\jumpl v \jumpr = (v-g)\boldsymbol{n}$ and $\jumpl\boldsymbol{w}\jumpr = (\boldsymbol{w}-\boldsymbol{g})\cdot \boldsymbol{n}$;
    \item the average operator $\averagel\cdot\averager$ on $F\in\facesD$: $\averagel v \averager = v$ and $\averagel \boldsymbol{w} \averager = \boldsymbol{w}$.
\end{itemize}

\subsection{PolyDG semi-discrete formulation}
To derive the semi-discrete formulation of the problem, we introduce the discontinuity penalization function $\gamma_F$:
\begin{equation}
\label{penltycoeff}
  \gamma_F = \gamma_0 
  \begin{dcases}
    \max\left\{\{d^K\}_\mathrm{H},\{k^K\}_\mathrm{H}\right\}\dfrac{p^2}{\{h_K\}_\mathrm{H}}, & \mathrm{on} \; F \in \facesinternal, \\
    \max\left\{d^K, k^K\right\}\dfrac{p^2}{h_K},                               & \mathrm{on} \; F \in \facesD,
  \end{dcases}
\end{equation}
where $\gamma_0$ is a constant parameter that should be chosen sufficiently large to ensure the stability of the discrete formulation, $d^K = \|\sqrt{\mathbf{D}|_K}\|^2$, and $k^K = \|\,(1 + k_{12}|_K)(k_1|_K + \tilde{k}_1|_K)\|$. In Equation \eqref{penltycoeff}, the symbol $\{\cdot\}_\mathrm{H}$ denotes the the harmonic average operator, which is defined as $\{\omega^K\}_\mathrm{H} = \frac{2\omega^{K^+} \omega^{K^-}}{\omega^{K^+}+\omega^{K^-}}$ for any function $\omega^K$ associated with the partition $\partition$.
We define for any $u_h, w_h, v_h \in \Wh$ the following bilinear and trilinear forms:
\begin{subequations}
\begin{gather}
\mathcal{A}(u_h,v_h) = \int_{\Omega} \left (\mathbf{D}  \nabla_h u_h \right) \cdot \nabla_h v_h + \sum_{F \in \facesinternal \cup \facesD} \int_{F} \left(\gamma_F \jumpl u_h \jumpr \cdot \jumpl v_h \jumpr - \averagel \mathbf{D}  \nabla_h u_h \averager \cdot \jumpl v_h \jumpr - \jumpl u_h \jumpr \cdot \averagel \mathbf{D}  \nabla_h v_h \averager \right), \\
r_{N}(u_h,w_h,v_h) = \int_\Omega k_{12}\, u_h \, w_h \, v_h, \qquad
r_{L}(u_h,v_h) = \int_\Omega k_{1}\, u_h \, v_h, \qquad
\widetilde{r}_{L}(u_h,v_h) = \int_\Omega \widetilde{k}_{1}\, u_h \, v_h, \\
F_c(v_h) = \int_\Omega f_c \, v_h, \qquad
F_q(v_h) = \int_\Omega f_q \, v_h,
\end{gather}
\end{subequations}    
where we denote with $c_{\mathrm{D}h}$ and $q_{\mathrm{D}h}$ suitable approximations of $c_{\mathrm{D}}$ and $q_{\mathrm{D}}$ in $\Wh$, respectively. Moreover, we are using the elementwise gradient $\nabla_h \cdot$ \cite{quarteroni:EDP}. Finally, the semi-discrete formulation of the problem reads:
\par
Given suitable discrete approximations $c_{0h},q_{0h}\in \Wh$ of the initial conditions of Equation~\eqref{eq:HM_governingequation}, Find $(c_h(t),q_h(t))\in \Wh\times\Wh$ for each $t>0$ such that:   
\begin{equation}
\label{eq:weakformulation}
\begin{dcases}
  \left(\dfrac{\partial c_h}{\partial t}, w_h\right)_{\Omega} + \mathcal{A}(c_h,w_h) + r_L(c_h,w_h) + r_N(c_h,q_h,w_h) = F_c(w_h) & \forall w_h\in \Wh, \\
 \left(\dfrac{\partial q_h}{\partial t}, v_h\right)_{\Omega} + \mathcal{A}(q_h,v_h) + \widetilde{r}_L(q_h,v_h) - r_N(q_h,c_h,v_h) = F_q(v_h) & \forall v_h\in \Wh.
\end{dcases}
\end{equation}
\begin{remark}
    The dependency of the penalty coefficient expression on the reaction coefficients in Equation \eqref{penltycoeff} is justified by the forthcoming analysis of Theorem \ref{APrioriErrorEstimate}, which provides a relation between the coercivity constant and the reaction parameters to be respected.
\end{remark}
\section{Analysis of the semi-discrete PolyDG formulation}
\label{Analysis}
Before delving into the main analyses, we present a set of preliminary definitions and recall technical results that are instrumental for the forthcoming analysis. We define the following DG-norms:
\begin{subequations}
\begin{align}
     \| c \|^2_\mathrm{DG} = & \left\| \sqrt{\mathbf{D}}\nabla_h c \right\|^2 +  \left\| \sqrt{\gamma_F} \jumpl c \jumpr \right\|^2_{\facesinternal \cup \facesD} \qquad \forall c\in W_1, \\
     \tnorm{c}\tnorm_\mathrm{DG}^2 = & \| c \|_\mathrm{DG}^2 + \left\| \gamma_F^{-1/2} \averagel \mathbf{D}\nabla_h c \averager \right\|_{ \facesinternal \cup \facesD}^2 \qquad \forall c\in W_2,
\end{align}
\end{subequations}
where $W_1$ and $W_2$ are the broken Sobolev spaces defined as Equation \eqref{eq:brokensob}.
We recall that the previous norms are equivalent in the space of discontinuous functions $\Wh$, under Assumption~\ref{Mesh_quality_assumption} thanks to the following trace-inverse inequality \cite{cangiani_hp-version_2014}:
\begin{equation}
     \label{trace_inverse}
     \| c \|^2_{L^2(\partial K)} \lesssim \frac{p^2}{h_K}\| c \|^2_{L^2(K)} \quad \forall c \in \mathbb{P}_p (K)\quad \forall K\in\partition.
\end{equation}
We define also the energy norms:
\begin{equation}
\label{eq:energynorm}
\| c(t) \|_{\varepsilon}^2 = \| c(t) \|^2 + \int_{0}^{t} \left(\| c(s) \|_\mathrm{DG}^2 + \| c(s) \|^2\right) \mathrm{d}s,  \qquad
\tnorm c(t)\tnorm_{\varepsilon}^2  = \| c(t) \|^2 + \int_{0}^{t} \left(\tnorm c(s)\tnorm_\mathrm{DG}^2 + \| c(s) \|^2\right) \mathrm{d}s. 
\end{equation}
\subsection{Preliminary technical results}
It can be proven \cite{cangianiHpVersionDiscontinuousGalerkin2017} that the bilinear form $\mathcal{A}(\cdot,\cdot)$ is continuous and coercive, namely:
\begin{proposition}
    \label{ContinuiutyCoercivityA}
    Let Assumption~\ref{Mesh_quality_assumption} be satisfied, then the bilinear form $\mathcal{A}(\cdot,\cdot)$ is such that:
    \begin{itemize}
        \item  $|\mathcal{A}(u_h,v_h)| \lesssim \| u_h \|_\mathrm{DG}\| v_h \|_\mathrm{DG} \quad \forall u_h,v_h \in \Wh$;
        \item  $|\mathcal{A}(u,v_h)| \lesssim \tnorm u\tnorm_\mathrm{DG}\| v_h \|_\mathrm{DG} \quad \forall u \in W_2, v_h \in \Wh$;
        \item $ \mathcal{A}(u_h,u_h) \geq \mu \|u_h\|_\mathrm{DG}^2 \quad \forall u_h \in \Wh$; 
    \end{itemize}
    where $\mu$ is independent of $h$. The third bound holds provided that the penalty parameter $\gamma_0$ defined in Equation~\eqref{penltycoeff} is chosen large enough.
\end{proposition} 
We present two relevant inequalities that play a crucial role in deriving the stability estimate of the discrete solution.
\begin{proposition}
     \label{Gagliardo-Nirenberg} (Discrete Gagliardo-Nirenberg inequality)
      Under Assumption~\ref{Mesh_quality_assumption},
      for any $u_h \in \Wh$, there exists a constant $C_{\mathrm{G}_d}$, eventually depending on $p$ but independent on $h$, such that:
     \begin{equation*}
         \| u_h \|_{L^q(\Omega)} \leq C_{\mathrm{G}_d} \| u_h \|_\mathrm{DG}^s\,\| u_h \|^{1-s}
     \end{equation*}
     with $s\in(0,1]$ and $q$ satisfying:
     \begin{equation*}
         \frac{1}{q}= s \left( \frac{1}{2}-\frac{1}{d}\right) +\frac{1-s}{2}.
     \end{equation*}
\end{proposition}
The discrete Gagliardo-Nirenberg inequality can be found with complete proof in \cite{nirenbergdiscrete}. For brevity, we present the Perov inequality without including the proof, which can be found in \cite{Perov:inequality}:
\begin{proposition}
\label{Perov_inequality}
  Let $a,b$ be two positive constants, let $\eta>1$, and let $u\in L^\infty_+(0,\hat{t})=\{ u\in L^\infty(0,\hat{t}): \; u(t)\geq0\;a.e.\,\}$ such that:
    \begin{equation}
        u(t) \leq a + b \int_0^t u^\eta(s) \mathrm{d}s, \qquad \mathrm{for\;almost\;any}\; t\in(0,\hat{t}),
    \end{equation}
where $\hat{t}$ is such that:
\begin{equation}
    \hat{t} < \dfrac{\eta-1}{a^{\eta-1}b}.
    \label{eq:timebound}
\end{equation}
Then for almost any $t\in(0,\hat{t})$ we have:
\begin{equation}
    u(t) \leq \dfrac{a}{\left(1-a^{\eta-1}\,b\,(\eta-1)t\right)^\frac{1}{\eta-1}}.
\end{equation}
\end{proposition}
\subsection{Stability analysis of the semidiscrete formulation}
We can now proceed to prove a stability estimate for the discrete formulation~\eqref{eq:weakformulation}.

\begin{theorem}
    \label{StabilityEstimate}
    Assume that Assumptions~\ref{CoefficientsRegularity} and~\ref{Mesh_quality_assumption} are satisfied, and assume that the stability parameter $\gamma_0$ defined in Equation~\eqref{penltycoeff} has been chosen sufficiently large. Let $(c_h(t),q_h(t))\in W_h^{DG}\times W_h^{DG}$ be the solution of Problem (\ref{eq:weakformulation}) with $c_\mathrm{D}=q_\mathrm{D}=0$ for any $t\in(0,\widetilde{t}]$, with $\widetilde{t}\le T$ introduced in Proposition~\ref{Perov_inequality}. Then:
    \begin{equation}
        \label{eq:StabilityEstimate}
        \| c_h \|_{\varepsilon}^2+ \| q_h \|_{\varepsilon}^2 \leq \dfrac{\left(\| c_{0h} \|^2+ \| q_{0h} \|^2 + \displaystyle\int_{0}^{T} \left(\| f_c(s) \|^2+ \| f_q(s) \|^2\right) \mathrm{d}s \right)}{\left(\widetilde{\mu}^{d-1}- \widetilde{\mu}^{-1}\xi t\left(\| c_{0h} \|^2+ \| q_{0h} \|^2 + \displaystyle\int_{0}^{T} \left(\| f_c(s) \|^2+ \| f_q(s) \|^2\right)\right)^{d-1}\right)^{\frac{1}{d-1}}}
    \end{equation}
    where:
    \begin{align*}
        \widetilde{\mu}= \min\{1,\mu,K_1,\tilde{K}_1\}
        && \xi = \frac{C_{\mathrm{G}_d}^3 K_{12}^2}{2^{d-2}\mu} 
    \end{align*}
    having defined $K_{12}= \| k_{12} \|_{L^\infty}$, $K_1 = \| k_1 \|_{L^\infty}$, $\widetilde{K}_1 = \| \tilde{k}_1 \|_{L^\infty}$, and $C_{\mathrm{G}_d}$ from Proposition~\ref{Gagliardo-Nirenberg}.
\end{theorem}
\begin{proof}
We sum the two equations in~\eqref{eq:weakformulation} and we introduce the notation $\dot{u}=\frac{\partial u}{\partial t}$ for each $u\in\Wh$:
\begin{equation}
\label{eq:WK_uniqueeq}
\begin{split}
\left(\dot c_h, w_h\right)_\Omega + & \left(\dot q_h, v_h\right)_\Omega + \mathcal{A}(c_h,w_h) + \mathcal{A}(q_h,v_h) + r_{L}(c_h,w_h) +\widetilde{r}_L(q_h,v_h) = \\ &  - r_{N}(c_h,q_h,w_h)  + r_{N}(q_h,c_h,v_h) + F_c(w_h) + F_q(v_h).
\end{split}
\end{equation}
By choosing as trial functions $w_h=c_h(t)$ and $v_h=q_h(t)$, we integrate in time Equation~\eqref{eq:WK_uniqueeq} and exploit Proposition~\ref{ContinuiutyCoercivityA} to get:
\begin{equation*}
\begin{split}
\frac{1}{2} \| c_h(t) & \|^2 -  \frac{1}{2}\| c_{0h} \|^2 + \frac{1}{2}  \| q_h(t) \|^2 - \frac{1}{2}\| q_{0h} \|^2 + \int_{0}^{t} \left(\mu \left(\| c_h(s) \|_\mathrm{DG}^2 + \| q_h(s) \|_\mathrm{DG}^2\right) + \left(K_1 \|c_h(s)\| + \widetilde{K}_1 \|q_h(s)\|\right)\right) \mathrm{d}s \leq \\ - & \int_{0}^{t} r_{N}(c_h(s),q_h(s),c_h(s)) \mathrm{d}s + \int_{0}^{t} r_{N}(c_h(s),q_h(s),c_h(s)) \mathrm{d}s + \int_{0}^{t} F_c(c_h(s))\mathrm{d}s + \int_{0}^{t} F_q(q_h(s))\mathrm{d}s.
\end{split}
\end{equation*}
On the other hand, thanks to coefficients regularity (\ref{CoefficientsRegularity}) and H\"older's and Young's inequalities, we obtain:
\begin{equation*}
\begin{split}
\| c_h(t) \|^2 - & \| c_{0h} \|^2 + \| q_h(t) \|^2 - \| q_{0h} \|^2 + \int_{0}^{t} \left(2\mu \left(\| c_h(s) \|_\mathrm{DG}^2 + \| q_h(s) \|_\mathrm{DG}^2\right) + \left(K_1 \|c_h(s)\| + \widetilde{K}_1 \|q_h(s)\|\right)\right) \mathrm{d}s \leq \\  + &  \frac{1}{K_1}\int_{0}^{t}\| f_c(s) \|^2\mathrm{d}s + \frac{1}{\widetilde{K_1}}\int_{0}^{t}\| f_q(s) \|^2\mathrm{d}s + \int_0^t  2|r_{N}(c_h(s),q_h(s),c_h(s)) + r_{N}(q_h(s),c_h(s),q_h(s))| \mathrm{d}s,
\end{split}
\end{equation*}
where $K_1=\|k_1\|_{L^{\infty}(\Omega)}$ and $\widetilde{K}_1=\|\tilde{k}_1\|_{L^{\infty}(\Omega)}$. Concerning the non-linear terms we apply first the H\"older's inequality to obtain:
\begin{equation*}
    \int_{0}^{t} | r_N(c_h(s),q_h(s),c_h(s)) + r_N(q_h(s),c_h(s),q_h(s))| \mathrm{d}s \leq K_{12} \int_{0}^{t} \left(\int_{\Omega} | c_h^2(s)q_h(s) | + | q_h^2(s)c_h(s) |\right)\mathrm{d}s,
\end{equation*}
where $K_{12}=\| k_{12} \|_{L^{\infty}(\Omega)}$. Then we employ H\"older's inequality  with $\gamma = 3/2 $ and $\gamma' = 3$ on each term and get:
\begin{equation*}
   \int_{\Omega}\left(| c_h^2 q_h | + | c_h q^2_h |\right)
    \le \left(\int_{\Omega} c_h^3 \right)^{2/3}\left(\int_{\Omega}  q_h^3\right)^{1/3} + \left(\int_{\Omega} c_h^3 \right)^{1/3}\left(\int_{\Omega}  q_h^3\right)^{2/3}
    \le \| c_h \|_{L^3}^2 \|  q_h \|_{L^3} +  \| c_h \|_{L^3} \|  q_h \|^2_{L^3}.
\end{equation*}
Subsequently we apply the Young's inequality with exponents $\alpha = 3/2 $ and $\beta = 3$ and then sum up the two contributions:
\begin{equation*}
    \int_{0}^{t} | r_N(c_h(s),q_h(s),c_h(s)) + r_N(q_h(s),c_h(s),q_h(s))| \mathrm{d}s \leq K_{12} \int_{0}^{t} \left( \| c_h(s) \| _{L^3}^3 + \| q_h(s) \| _{L^3}^3\right) \mathrm{d}s.
\end{equation*}
In order to exploit the Gagliardo-Nirenberg inequality in Proposition~\ref{Gagliardo-Nirenberg}, we distinguish 2 cases:
\begin{itemize}
    \item $d=2$, then $s=1/3$ and the Young's inequality can be applied with exponents $\alpha =\beta = 2$ and with an $\epsilon>0$:
    \begin{equation*}
        \| c_h \|_{L^3}^3 \le \frac{C_{\mathrm{G}_d}^3}{2}\left(\epsilon \| c_h \|_\mathrm{DG}^2+\frac{1}{\epsilon}\| c_h \|^4\right);
    \end{equation*}
    \item $d=3$, then $s=1/2$ and Young's inequality can be applied with exponents  $\alpha = 4/3$, $\beta = 4$ and with an $\epsilon>0$:
    \begin{equation*}
         \| c_h \|_{L^3}^3 \le \frac{C_{\mathrm{G}_d}^3}{4}\left(3\epsilon \| c_h \|_\mathrm{DG}^2+\frac{1}{\epsilon}\| c_h \|^6\right).
    \end{equation*}
\end{itemize}
Then by choosing $\epsilon = 2\mu/(d K_{12} C^3_{\mathrm{G}_d})$, we obtain:
\begin{equation}
\begin{split}
  \tilde{\mu} \Big( \| c_h(t) \|^2 + & \| q_h(t) \|^2 + \int_{0}^{t} \left(\| c_h(s) \|_{\mathrm{DG}}^2+ \| q_h(s) \|_{\mathrm{DG}}^2 \| c_h(s) \|^2 + \| q_h(s) \|^2 \right) \mathrm{d}s \Big) \le \| c_{0h} \|^2 + \| q_{0h} \|^2 \\ + &
\int_{0}^{T} \left(\| f_c(s) \|^2+ \| f_q(s) \|^2\right) \mathrm{d}s + \frac{d K^2_{12} C_{\mathrm{G}_d}^6 }{2^{d-3} \mu}\int_{0}^{t}\left(\| c_h(s) \|^{2d} + \| q_h(s) \|^{2d}\right) \mathrm{d}s.
\end{split}
\end{equation}
Finally, we can define a positive constant $\widetilde{\mu} = \min\{1,2\mu,K_1,\tilde{K}_1\}$ and by applying Perov's inequality in Proposition~\ref{Perov_inequality}, we deduce the thesis.
\end{proof}
\begin{remark}
Under the assumptions of Theorem~\ref{StabilityEstimate}, if $f_c=f_q=0$, estimate (\ref{eq:StabilityEstimate}) reduces to:
\begin{equation}
    \label{StabilityEstimate_HomofeneousF}
        \| c_h \|_{\varepsilon}^2+ \| q_h \|_{\varepsilon}^2 \leq \dfrac{\left(\| c_{0h} \|^2+ \| q_{0h} \|^2 + T \| k_0 \|^2 \right)}{\left(\widetilde{\mu}^{d-1}- \widetilde{\mu}^{-1}\xi t\left(\| c_{0h} \|^2+ \| q_{0h} \|^2 + T \| k_0 \|^2\right)^{d-1}\right)^{\frac{1}{d-1}}}=C_S
\end{equation}
\end{remark}
\begin{remark}
Comparing the estimate of Theorem~\ref{StabilityEstimate} with the one derived for FK model in \cite{corti_discontinuous_2023}, we can observe a significant improvement, due to the dependence on time, which is no more exponential. This fact ensure stability for much longer time $T$. 
\end{remark}
\subsection{Error analysis of the semidiscrete formulation}
\begin{theorem}
\label{APrioriErrorEstimate}
Let problem~\eqref{eq:HM_governingequation} with $f_c=f_q=0$ and $c_\mathrm{D}=q_\mathrm{D}=0$. Let $(c(t),q(t))$ be the solution of Problem~\eqref{eq:HM_governingequation} for any $t\in(0,T]$ and assume that $c,q \in C^1([0,T],H^n(\Omega))$ with $\;n\geq 2$. Under Assumptions~\ref{CoefficientsRegularity} and~\ref{Mesh_quality_assumption}, let $(c_h(t),q_h(t))$ be the solution of~\eqref{eq:weakformulation}, with a sufficiently large penalty parameter $\gamma_0$. Then, the following estimate holds:  
\begin{equation*}
\begin{split}
  \tnorm c(t)-c_h(t)\tnorm_{\varepsilon}^2 + \tnorm q(t)-q_h(t)\tnorm_{\varepsilon}^2 \lesssim & \sum_{K\in\partition} h_K^{2\min(p+1,n)-2} \Big( \| c(t) \|_{H^n(K)}^2 +\| q(t) \|_{H^n(K)}^2 \\
  + & \int_{0}^{t}\left(\| c(s) \|_{H^n(K)}^2 + \| \dot c(s) \|_{H^n(K)}^2 +\| q(s) \|_{H^n(K)}^2 + \|\dot q(s) \|_{H^n(K)}^2\right)\mathrm{d}s \Big).
\end{split}
\end{equation*}
\end{theorem}
\begin{proof}
Let us consider the weak solution of the heterodimer model ($c$, $q$), then due to the consistency of the DG formulation, it holds:
\begin{equation}
  \label{eq:DG_exact}
  \left(\dot c, w_h\right)_{\Omega}+\left(\dot q, v_h\right)_{\Omega} + \mathcal{A}(c,w_h) + \mathcal{A}(q,v_h) + r_{L}(c,w_h) + \widetilde{r}_L(q,v_h) = - r_N(c,q,w_h) + r_N(q,c,v_h) + (k_0,w_h)_\Omega.
\end{equation}
We subtract Equations~\eqref{eq:WK_uniqueeq} and~\eqref{eq:DG_exact} and consider the interpolant of the exact solutions $c_{\mathrm{I}},q_{\mathrm{I}}\in\Wh$ \cite{babuska-interpolant}, respectively, then rewrite the errors as:
\begin{equation*}
    c-c_h = (c-c_\mathrm{I})+(c_\mathrm{I}-c_h) = e_{\mathrm{I}}^c+e_h^c, \qquad q-q_h = (q-q_\mathrm{I})+(q_\mathrm{I}-q_h) = e_{\mathrm{I}}^q+e_h^q. 
\end{equation*}
We test the obtained formulation with respect to $w_h=e_h^c$ and $v_h=e_h^q$ and integrate along the time interval $[0,t]$:

\begin{equation*}
\begin{split}
    \int_{0}^{t} & \left((\dot e_h^c(s), e_h^c(s))_\Omega + (\dot e_h^q(s), e_h^q(s))_\Omega + \mathcal{A}(e_h^c(s),e_h^c(s)) + \mathcal{A}(e_h^q(s),e_h^q(s))  +r_L(e_h^c(s),e_h^c(s)) + \widetilde{r}_L(e_h^q(s),e_h^q(s)) 
 \right)\mathrm{d}s =  \\
  \int_{0}^{t} & \Big( - r_N(c(s),q(s),e_h^c(s)) + r_N(c_h(s),q_h(s),e_h^c(s)) + r_N(q(s),c(s),e_h^q(s)) - r_N(q_h(s),c_h(s),e_h^q(s)) \\
  - & (\dot e_{\mathrm{I}}^c(s), e_h^c(s))_\Omega - (\dot e_{\mathrm{I}}^q(s), e_h^q(s))_\Omega - \mathcal{A}(e_{\mathrm{I}}^c(s),e_h^c(s)) - \mathcal{A}(e_{\mathrm{I}}^q(s),e_h^q(s)) - r_L(e_{\mathrm{I}}^c(s),e_h^c(s)) 
  - \widetilde{r}_L(e_{\mathrm{I}}^q(s),e_h^q(s))\Big)\mathrm{d}s.
\end{split}
\end{equation*}
We proceed by exploiting the coercivity and continuity of the bilinear form $\mathcal{A}$ in Proposition~\ref{ContinuiutyCoercivityA}, and by means of H\"older's and Young's inequalities:
\begin{equation*}
\begin{split}
     \|e_h^c(t)\|^2 + & \|e_h^q(t)\|^2 +\int_{0}^{t} \left(\|e_h^c(s)\|_\mathrm{DG}^2 + \|e_h^q(s)\|_\mathrm{DG}^2+ \|e_h^c(s)\|^2 + \|e_h^q(s)\|^2\right)\mathrm{d}s \lesssim \\ 
     + & \int_{0}^{t} \left(|\left(c(s)q(s)-c_h(s)q_h(s),e_h^c(s)\right)_\Omega | + |\left(c(s)q(s)-c_h(s)q_h(s),e_h^q(s)\right)_\Omega |\right)\mathrm{d}s \\
     + & \int_{0}^{t} \left(\left(\|\dot e_{\mathrm{I}}^c(s) \| + \| e_{\mathrm{I}}^c(s) \|\right)\|e_h^c(s) \| + \left(\|\dot e_{\mathrm{I}}^q(s) \| + \| e_{\mathrm{I}}^q(s) \|\right)\|e_h^q(s) \| \right)  \mathrm{d}s \\
    + & \int_{0}^{t} \left( \tnorm e_{\mathrm{I}}^c(s)\tnorm_\mathrm{DG}\|e_h^c(s)\|_\mathrm{DG} + \tnorm e_{\mathrm{I}}^q(s)\tnorm_\mathrm{DG}\|e_h^q(s)\|_\mathrm{DG}\right) \mathrm{d}s.
     \end{split}
\end{equation*}
Concerning the nonlinear term, we observe that:
\begin{equation*}
    \begin{split}
    cq-c_hq_h = cq -qc_{\mathrm{I}} + qc_{\mathrm{I}} -q_hc_{\mathrm{I}} +q_hc_{\mathrm{I}}-q_hc_h  \\
    =q\underbrace{(c-c_{\mathrm{I}})}_{e_{\mathrm{I}}^c} + c_{\mathrm{I}}\underbrace{(q-q_h)}_{e_{\mathrm{I}}^q+e_h^q}+ q_h\underbrace{(c_{\mathrm{I}}-c_h)}_{e_{h}^c}.
    \end{split}
\end{equation*}
Therefore:
\begin{equation*}
    |\left(cq-c_hq_h,e_h^c\right)_\Omega|= \underbrace{|\left(qe_{\mathrm{I}}^c,e_h^c\right)_\Omega|}_{(\mathrm{I})} + \underbrace{|\left(c_{\mathrm{I}}e_{\mathrm{I}}^q,e_h^c\right)_\Omega|}_{(\mathrm{II})} + \underbrace{|\left(c_{\mathrm{I}}e_{h}^q,e_h^c\right)_\Omega|}_{(\mathrm{III})} + \underbrace{|\left(q_he_{h}^c,e_h^c\right)_\Omega|}_{(\mathrm{IV})}.
\end{equation*}
The terms on the right-hand side above can be treated separately as follows:

\begin{enumerate}
     \item[(I)] Using that the exact solution $q(t)\in L^{\infty}(\Omega)$ for any $t>0$, H\"older's and Young's inequalities to obtain:
     \begin{equation*}
     |\left(q\, e_{\mathrm{I}}^c,e_h^c\right)_\Omega| \lesssim \|e_{\mathrm{I}}^c \| \| e_h^c \| ;
     \end{equation*}
     \item[(II)] Using that the interpolant is bounded by construction for any $t>0$, H\"older's and Young's inequalities lead to:
     \begin{equation*}
    |\left(c_\mathrm{I}\,e_{\mathrm{I}}^q,e_h^c\right)_\Omega| \lesssim  \|e_{\mathrm{I}}^q \| \| e_h^c \|;
     \end{equation*}
     \item[(III)] We exploit the interpolant boundedness, H\"older's, and Young's inequalities to achieve:
     \begin{equation*}
          |\left(c_\mathrm{I}\,e_h^q,e_h^c\right)_\Omega| \lesssim \|e_h^q \|^2+ \| e_h^c \|^2;
     \end{equation*}
     \item[(IV)] We exploit the stability estimate \eqref{eq:StabilityEstimate}, discrete Gagliardo-Nirenberg inequality \cite{nirenbergdiscrete}, and Young inequality with $\epsilon>0$ small enough to obtain:
    \begin{equation*}
         |\left(q_h\,e_h^c,e_h^c\right)_\Omega| \lesssim \|e_h^c \|^2_{L^4(\Omega)} \lesssim \epsilon \|e_h^c \|^2_\mathrm{DG} + \dfrac{1}{\epsilon}\|e_h^c \|^2;
     \end{equation*}
     Notice that the Gagliardo-Nirenberg inequality applies in $d=2$ with $s=1/2$ and then employs Young inequality with exponents $\alpha=\beta=2$. Concerning $d=3$, we apply the Gagliardo-Nirenberg inequality with $s=3/4$ and employ Young inequality with exponents $\alpha=4/3$ and $\beta=4$.   
\end{enumerate}
Analogous considerations can be done to treat the second nonlinear term. Then we obtain:
\begin{equation*}
\begin{split}
     \tnorm e_h^c(t)\tnorm^2_\varepsilon & + \tnorm e_h^q(t)\tnorm^2_\varepsilon \lesssim \int_{0}^{t} \left(\|e_h^c(t)\|^2 + \|e_h^q(t)\|^2 \right)\mathrm{d}s\\ 
     + & \int_{0}^{t}  \left(\|\dot e_{\mathrm{I}}^c(s) \|^2 + \| e_{\mathrm{I}}^c(s) \|^2 + \tnorm e_{\mathrm{I}}^c(s)\tnorm_\mathrm{DG}^2 + \| \dot e_{\mathrm{I}}^q(s) \|^2 + \| e_{\mathrm{I}}^q(s) \|^2 + \tnorm e_{\mathrm{I}}^c(s)\tnorm_\mathrm{DG}^2 \right) \mathrm{d}s.
     \end{split}
\end{equation*}  
By application of H\"older's inequality and  Gr\"onwall's lemma \cite{quarteroni:EDP} we get: 
\begin{equation*}
     \tnorm e_h^c(t)\tnorm^2_\varepsilon + \tnorm e_h^q(t)\tnorm^2_\varepsilon \lesssim e^t \int_{0}^{t} \left(\|\dot e_{\mathrm{I}}^c(s) \|^2 + \| e_{\mathrm{I}}^c(s) \|^2 + \tnorm e_{\mathrm{I}}^c(s)\tnorm_\mathrm{DG}^2 + \| \dot e_{\mathrm{I}}^q(s) \|^2 + \| e_{\mathrm{I}}^q(s) \|^2 + \tnorm e_{\mathrm{I}}^c(s)\tnorm_\mathrm{DG}^2 \right) \mathrm{d}s.
\end{equation*}
We can bound each term with the interpolation error estimate \cite{riviereDiscontinuousGalerkinMethods2008}:
\begin{equation*}
\tnorm e_h^c\tnorm_{\varepsilon}^2 + \tnorm e_h^q \tnorm_{\varepsilon}^2 \lesssim \sum_{K\in\partition} h_L^{2\min(p+1,n)-2} \int_{0}^{t}(\|  c(s) \|_{H^n(\mathrm{K})}^2 + \|  \dot c(s) \|_{H^n(\mathrm{K})}^2 +\| q(s) \|_{H^n(\mathrm{K})}^2 + \|\dot q(s) \|_{H^n(\mathrm{K})}^2).
\end{equation*}
Finally, by application of the triangular inequality:
\begin{equation*}
   \tnorm c-c_h\tnorm_{\varepsilon}^2+\tnorm q-q_h\tnorm_{\varepsilon}^2 \le \tnorm e_h^c\tnorm_{\varepsilon}^2+\tnorm e_{\mathrm{I}}^c\tnorm_{\varepsilon}^2+ \tnorm e_h^q\tnorm_{\varepsilon}^2+\tnorm e_{\mathrm{I}}^q\tnorm_{\varepsilon}^2,
\end{equation*}
and taking into account the interpolation error estimate, the thesis follows.
\end{proof}
\begin{remark}
Comparing the estimate of Theorem~\ref{APrioriErrorEstimate} with the one derived for FK model in \cite{corti_discontinuous_2023}, we can observe a significant improvement, due to the absence of relations between the model coefficients to be satisfied. 
\end{remark}
\section{Fully discrete formulation}
\label{TimeDiscret}
In this section, we present the fully discrete formulation of Problem~\eqref{eq:weakformulation} in its algebraic form. We consider $\{\phi_j\}_{j=0}^{N_h}$ a set of basis function for $\Wh$, being $N_h$ its dimension. We can write the discrete solutions expansion with respect to the chosen base, i.e. $c_h(\boldsymbol{x},t) = \sum_{j=0}^{N_h} \mathbf{C}_j(t) \phi_j(\boldsymbol{x})$ and $q_h(\boldsymbol{x},t) = \sum_{j=0}^{N_h} \mathbf{Q}_j(t) \phi_j(\boldsymbol{x})$. We denote by $\mathbf{C}$, $\mathbf{Q}\in\mathbb{R}^{N_h}$, the corresponding expansion coefficients vectors. We also define the following matrices

\begin{align*}
     [\mathrm{M}]_{i,j}                 & = (\phi_j,\phi_i)_\Omega          & \mathrm{Mass\;matrix} \\
     [\mathrm{A}]_{i,j}                 & = \mathcal{A}(\phi_j,\phi_i)      & \mathrm{Stiffness\;matrix} \\
     [\mathrm{R}_N(\Phi(t))]_{i,j}      & = r_N(\Phi_h(t),\phi_j,\phi_i)    & \mathrm{Non-linear\;reaction\;matrix} \\
     [\mathrm{R}_L]_{i,j}               & = r_L(\phi_j,\phi_i)              & \mathrm{Linear\;reaction\;matrix} \\
     [\widetilde{\mathrm{R}}_L]_{i,j}   & = \widetilde{r}_L(\phi_j,\phi_i)  & \mathrm{Linear\;reaction\;matrix}
\end{align*}
with $i,j = 1,...,N_h$. Moreover, we introduce the right-hand side vectors : $[\mathbf{F}_c]_i= F_c(\phi_i)$ and $[\mathbf{F}_q]_i= F_q(\phi_i)$ with $i = 1,...,N_h$. We then rewrite the problem~\eqref{eq:weakformulation} in algebraic form: 
\begin{equation}
\label{eq:AlgebraiFormulation}
\begin{cases}
  \mathrm{M} \dot{\mathbf{C}}(t) + \mathrm{A}\mathbf{C}(t) + \mathrm{R}_{L}\mathbf{C}(t) + \mathrm{R}_{N}(\mathbf{C}(t))\mathbf{Q}(t)  = \mathbf{F}_c(t) &  t\in (0,T] \\[6pt]
 \mathrm{M}\dot{\mathbf{Q}}(t) + \mathrm{A}\mathbf{Q}(t) + \widetilde{\mathrm{R}}_{L}\mathbf{Q}(t) - \mathrm{R}_{N}(\mathbf{Q}(t))\mathbf{C}(t) = \mathbf{F}_q(t)  & t\in (0,T] \\[6pt]
 \mathbf{C}(0) = \mathbf{C}_0 \textit{,} \hspace{0.2cm} \mathbf{Q}(0) =\mathbf{Q}_0.
 \end{cases}
\end{equation} 
\par
We construct a partition $0 < t_1 < t_2 < ... < t_{N_T} = T$ of the time interval $[0,T]$ into $N_T$ intervals of constant time step $\Delta t = t_{n+1}-t_n$. We denote with $\mathbf{C}^n \simeq \mathbf{C}(t_n)$. We assemble the fully discrete approximation of Problem~\eqref{eq:AlgebraiFormulation} by means of the $\vartheta$-method, depending upon the user-set parameter $\vartheta \in [0,1]$.
\par
Given the initial conditions vectors $\mathbf{C}^0=\mathbf{C}(0)$ and $\mathbf{Q}^0=\mathbf{Q}(0)$, for $n= 1, ... , N_T$ find $(\mathbf{C}^{n}, \mathbf{Q}^{n})$ solution of the following system of equations:
\begin{equation}
    \label{eq:FullyDiscreteFormulation}
    \begin{dcases}
    \left(\frac{\mathrm{M}}{\Delta t} + \vartheta \mathrm{A} + \vartheta\mathrm{R}_{L}\right)\mathbf{C}^{n} +\mathrm{R}_{N}(\mathbf{C}^{*}) \mathbf{Q}^{n,n-1} = \mathbf{F}_{c}^{n}, \\  
    \left(\frac{\mathrm{M}}{\Delta t} + \vartheta \mathrm{A} + \vartheta\widetilde{\mathrm{R}}_{L}\right)\mathbf{Q}^{n} -\mathrm{R}_{N}(\mathbf{Q}^{*}) \mathbf{C}^{n,n-1} = \mathbf{F}_{q}^{n}, 
    \end{dcases}
\end{equation}
where: 
\begin{align*}
    \mathbf{F}_c^{n} = \vartheta \mathbf{F}_c(t_{n}) + (1-\vartheta)\mathbf{F}_c(t_{n-1}) + \frac{M}{\Delta t}\mathbf{C}^{n-1}-(1-\vartheta)(\mathrm{A}+\mathrm{R}_{L})\mathbf{C}^{n-1}, \\ 
    \mathbf{F}_q^{n} = \vartheta \mathbf{F}_q(t_{n}) + (1-\vartheta)\mathbf{F}_q(t_{n-1}) +  \frac{M}{\Delta t}\mathbf{Q}^{n-1}-(1-\vartheta)(\mathrm{A}+\widetilde{\mathrm{R}}_{L})\mathbf{Q}^{n-1}.
\end{align*}
We denoted with $\mathbf{C}^{*}$ and $\mathbf{Q}^{*}$ the approximation of $\mathbf{C}^{n}$ and $\mathbf{Q}^{n}$, chosen to linearize the system~\eqref{eq:FullyDiscreteFormulation}. In practice we will consider the following options:
\begin{itemize}
    \item Implicit Euler method $(\vartheta=1)$: 
    \begin{equation*}
        \mathrm{R}_{N}(\mathbf{Q}^{*}) \mathbf{C}^{n,n-1} = \mathrm{R}_{N}\left( \mathbf{Q}^{n-1}\right)\,\mathbf{C}^n;
    \end{equation*}
    \item Crank-Nicolson method $(\vartheta=1/2)$: 
    \begin{equation*}
        \mathrm{R}_{N}(\mathbf{Q}^{*}) \mathbf{C}^{n,n-1} = \mathrm{R}_{N} \left(\frac{3}{2}\mathbf{Q}^{n-1}-\frac{1}{2}\mathbf{Q}^{n-2}\right)\,\dfrac{1}{2}\left(\mathbf{C}^n+\mathbf{C}^{n-1}\right).
    \end{equation*}
\end{itemize}
\section{Numerical results: verification}
\label{Verification}
In this section, we present some numerical tests in order to assess the accuracy of the method, validate the error estimate proved in Section~\ref{Analysis}, and analyze the travelling-wave solution behavior. The numerical simulations are based on the Lymph library \cite{lymphpaper}, implementing the PolyDG method. Moreover, for all the subsequent tests, we set the penalty parameter $\gamma_0=10$.

\subsection{Test 1: convergence analysis}
We consider the domain $\Omega=(0,1)^2$ discretized into a polygonal mesh generated by PolyMesher \cite{Polymesher}. In the current phase of the study, we are focused on examining isotropic diffusion, and for the model parameters we adopt the values summarized in Table~\ref{tab:test1_param}.
\begin{table}[t]
	\centering
	\begin{tabular}{|c|r l|c|r l|}
	\hline
	\textbf{Parameter} & \multicolumn{2}{c|}{\textbf{Value}} &	\textbf{Parameter} & \multicolumn{2}{c|}{\textbf{Value}} \\ 
		\hline 
		 $d_\mathrm{ext}$     &  $1.00$       & $[\mathrm{mm^2/years}]$ 
		& $d_\mathrm{axn}$     &  $0.00$       & $[\mathrm{mm^2/years}]$ \\ 
		\hline 
		 $k_0$                 &  $0.00$       & $[\mathrm{1/years}]$ 
		& $k_{12}$              &  $1.00$       & $[\mathrm{1/years}]$ \\ 
		\hline 
		 $k_1$                 &  $1.00$       & $[\mathrm{1/years}]$ 
		& $\widetilde{k}_1$     &  $1.00$       & $[\mathrm{1/years}]$ \\ 
		\hline 
	\end{tabular}
	\caption{Physical parameter values used in Test 1 and Test 2.}
	\label{tab:test1_param}
\end{table}
\begin{figure}[t!]
    \begin{subfigure}[b]{0.5\textwidth}
            \resizebox{\textwidth}{!}{\definecolor{mycolor2}{rgb}{0.00000,1.00000,1.00000}%
\pgfplotsset{
  log x ticks with fixed point/.style={
      xticklabel={
        \pgfkeys{/pgf/fpu=true}
        \pgfmathparse{exp(\tick)}%
        \pgfmathprintnumber[fixed  zerofill, precision=2]{\pgfmathresult}
        \pgfkeys{/pgf/fpu=false}
      }
  }
}
\begin{tikzpicture}

\begin{axis}[%
width=3.875in,
height=2.36in,
at={(2.6in,1.099in)},
scale only axis,
xmode=log,
xmin=0.06,
xmax=0.32,
xminorticks=true,
xlabel = {$h$ [-]},
ylabel = {$||c(T)-c_h(T)||_\varepsilon$},
ymode=log,
ymin=1e-8,
ymax=0.01,
yminorticks=true,
axis background/.style={fill=white},
title style={font=\bfseries},
xmajorgrids,
xminorgrids,
ymajorgrids,
yminorgrids,
legend style={legend cell align=left, align=left, draw=white!15!black}
]
              
\addplot [color=red, line width=2.0pt]
  table[row sep=crcr]{%
    0.322736007494350  0.010095748022434\\
    0.181290869729279  0.004494085281393\\
    0.102590411249172  0.002088006049963\\
    0.056694075094424  9.60265754784e-04\\
};
\addlegendentry{$p=1$}

\addplot [color=orange, line width=2.0pt]
  table[row sep=crcr]{%
    0.322736007494350   0.004130645749828\\
    0.181290869729279	0.001180033545603\\
    0.102590411249172	3.83193972823e-04\\
    0.056694075094424	1.14835808279e-04\\
};
\addlegendentry{$p=2$}

\addplot [color=green, line width=2.0pt] 
  table[row sep=crcr]{%
    0.322736007494350	2.203364419614502e-04\\
    0.181290869729279   3.901439786658697e-05\\
    0.102590411249172	8.553707634901362e-06\\
    0.056694075094424   1.382446439947711e-06\\
};
\addlegendentry{$p=3$}

\addplot [color=mycolor2, line width=2.0pt]
  table[row sep=crcr]{%
    0.322736007494350   1.6047757709337e-05\\
    0.181290869729279	1.4637045837642e-06\\
    0.102590411249172	1.5446708621610e-07\\
    0.056694075094424	2.7219897900500e-08\\
};
\addlegendentry{$p=4$}

\node[right, align=left, text=black, font=\footnotesize]
at (axis cs:0.1005,0.001) {$1$};

\addplot [color=black, line width=1.5pt]
  table[row sep=crcr]{%
0.100   0.0012\\
0.075   0.0009\\
0.100   0.0009\\
0.100   0.0012\\
};

\node[right, align=left, text=black, font=\footnotesize]
at (axis cs:0.1005,0.00016) {$2$};

\addplot [color=black, line width=1.5pt]
  table[row sep=crcr]{%
0.100   0.0002\\
0.075   0.0001125\\
0.100   0.0001125\\
0.100   0.0002\\
};

\node[right, align=left, text=black, font=\footnotesize]
at (axis cs:0.1005,3.5e-6) {$3$};

\addplot [color=black, line width=1.5pt]
  table[row sep=crcr]{%
0.100   5.00e-06\\
0.075   2.11e-06\\
0.100   2.11e-06\\
0.100   5.00e-06\\
};

\node[right, align=left, text=black, font=\footnotesize]
at (axis cs:0.1005,5e-8) {$4$};

\addplot [color=black, line width=1.5pt]
  table[row sep=crcr]{%
0.100   1.000e-07\\
0.075   3.164e-08\\
0.100   3.164e-08\\
0.100   1.000e-07\\
};

\end{axis}
\end{tikzpicture}%
}
    \end{subfigure}%
    \begin{subfigure}[b]{0.5\textwidth}
            \resizebox{\textwidth}{!}{\definecolor{mycolor2}{rgb}{0.00000,1.00000,1.00000}%
\pgfplotsset{
  log x ticks with fixed point/.style={
      xticklabel={
        \pgfkeys{/pgf/fpu=true}
        \pgfmathparse{exp(\tick)}%
        \pgfmathprintnumber[fixed  zerofill, precision=2]{\pgfmathresult}
        \pgfkeys{/pgf/fpu=false}
      }
  }
}
\begin{tikzpicture}

\begin{axis}[%
width=3.875in,
height=2.36in,
at={(2.6in,1.099in)},
scale only axis,
xmode=log,
xmin=0.064,
xmax=0.3239,
xminorticks=true,
xlabel = {$h$ [-]},
ylabel = {$||q(T)-q_h(T)||_\varepsilon$},
ymode=log,
ymin=5e-5,
ymax=0.4,
yminorticks=true,
axis background/.style={fill=white},
title style={font=\bfseries},
xmajorgrids,
xminorgrids,
ymajorgrids,
yminorgrids,
legend style={legend cell align=left, align=left, draw=white!15!black}
]
              
\addplot [color=red, line width=2.0pt]
  table[row sep=crcr]{%
    0.322736007494350  0.320306290166234\\
    0.181290869729279  0.170072710281147\\
    0.102590411249172  0.065370648122349\\
    0.056694075094424  0.028541591009509\\
};

\addplot [color=orange, line width=2.0pt]
  table[row sep=crcr]{%
    0.322736007494350   0.238010552865124\\
    0.181290869729279	0.087174009937669\\
    0.102590411249172	0.033386623337429\\
    0.056694075094424	0.010242943501941\\
};

\addplot [color=green, line width=2.0pt]
  table[row sep=crcr]{%
    0.322736007494350	0.155294225770468\\
    0.181290869729279   0.027393322628307\\
    0.102590411249172	0.004527514592785\\
    0.056694075094424   7.068696699830621e-04\\
};

\addplot [color=mycolor2, line width=2.0pt]
  table[row sep=crcr]{%
    0.322736007494350   0.043224481278812\\
    0.181290869729279	0.003961584122848\\
    0.102590411249172	0.000474213701520\\
    0.056694075094424	0.000042768236460\\
};

\node[right, align=left, text=black, font=\footnotesize]
at (axis cs:0.1005,0.044) {$1$};

\addplot [color=black, line width=1.5pt]
  table[row sep=crcr]{%
0.100   0.0500\\
0.075   0.0375\\
0.100   0.0375\\
0.100   0.0500\\
};

\node[right, align=left, text=black, font=\footnotesize]
at (axis cs:0.1005,0.015) {$2$};

\addplot [color=black, line width=1.5pt]
  table[row sep=crcr]{%
0.100   0.02\\
0.075   0.01125\\
0.100   0.01125\\
0.100   0.02\\
};

\node[right, align=left, text=black, font=\footnotesize]
at (axis cs:0.1005,2e-3) {$3$};

\addplot [color=black, line width=1.5pt]
  table[row sep=crcr]{%
0.100   3e-03\\
0.075   1.265e-03\\
0.100   1.265e-03\\
0.100   3e-03\\
};

\node[right, align=left, text=black, font=\footnotesize]
at (axis cs:0.1005,1.5e-4) {$4$};

\addplot [color=black, line width=1.5pt]
  table[row sep=crcr]{%
0.100   3e-04\\
0.075   9.492e-05\\
0.100   9.492e-05\\
0.100   3e-04\\
};

\end{axis}
\end{tikzpicture}%
}
    \end{subfigure}%
    \caption{Test case 1: computed errors and computed convergence rates versus $h$ for different choices of the polynomial degree $p=1,2,3,4$ and for the solutions $c_h$ (left) and $q_h$ (right).}
    \label{fig:test1_errorsh}
\end{figure}
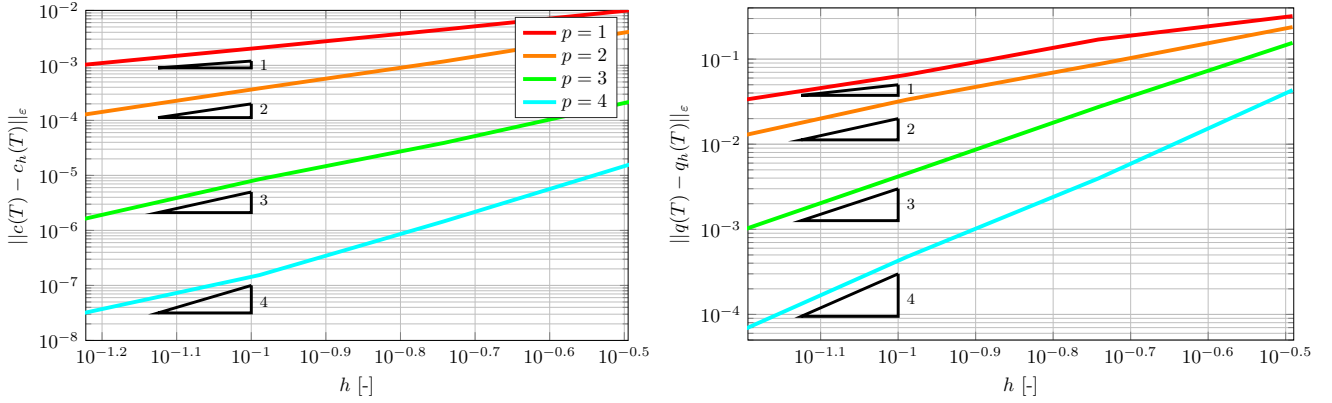
\begin{figure}[t!]
    \begin{subfigure}[b]{0.5\textwidth}
            \resizebox{\textwidth}{!}{\definecolor{mycolor1}{rgb}{1.00000,1.00000,0.00000}%
\definecolor{mycolor2}{rgb}{0.00000,1.00000,1.00000}%

\begin{tikzpicture}
\begin{axis}[%
width=3.875in,
height=2.30in,
at={(2.6in,1.099in)},
scale only axis,
xmin=1,
xmax=8,
xlabel style={font=\color{white!15!black}},
xlabel={$p$},
ymode=log,
ymin=1e-10,
ymax=0.50,
yminorticks=true,
ylabel style={font=\color{white!15!black}},
ylabel={Error},
axis background/.style={fill=white},
title style={font=\bfseries},
title={Convergence with respect to the polynomial degree $p$},
xmajorgrids,
xminorgrids,
ymajorgrids,
yminorgrids,
legend style={legend cell align=left, align=left, draw=white!15!black}
]

\addplot [color=purple, line width=2.0pt]
  table[row sep=crcr]{%
    1   0.010133540049389 \\
    2   0.004146592101946 \\
    3   2.21200501177e-04 \\
    4   1.61277166711e-05 \\
    5   5.76257883078e-07 \\
    6   5.58204982895e-08 \\
    7   8.11848872878e-09 \\
    8   6.19926811441e-10 \\
};
\addlegendentry{$||c(T)-c_h(T)||_\varepsilon$}

\addplot [color=purple, dashed, line width=2.0pt]
  table[row sep=crcr]{%
    1   0.321520388355109 \\
    2   0.238846692882545 \\
    3   0.155914729738798 \\
    4   0.043409052578839 \\
    5   0.013587830880721 \\
    6   0.002264909297743 \\
    7   4.88979267667e-04 \\
    8   5.61023298328e-05 \\
};
\addlegendentry{$||q(T)-q_h(T)||_\varepsilon$}

\end{axis}
\end{tikzpicture}%
}
            \caption{Test case 1: computed errors with respect to $p$.}
            \label{fig:test1_errorsp}
    \end{subfigure}%
    \begin{subfigure}[b]{0.5\textwidth}
            \resizebox{\textwidth}{!}{\definecolor{EI}{rgb}{0.00000,0.50000,1.00000}%
\definecolor{CN}{rgb}{1.00000,0.00000,0.50000}%
\begin{tikzpicture}

\begin{axis}[%
width=3.875in,
height=2.36in,
at={(2.6in,1.099in)},
scale only axis,
xmode=log,
xmin=0.025,
xmax=0.2000,
xminorticks=true,
xlabel = {$\Delta t$ [-]},
ylabel = {$||\cdot||_{\epsilon}$},
ymode=log,
ymin=1e-5,
ymax=0.05,
yminorticks=true,
axis background/.style={fill=white},
title style={font=\bfseries},
title={Convergence with respect to the timestep $\Delta t$},
xmajorgrids,
xminorgrids,
ymajorgrids,
yminorgrids,
legend style={at={(0.52,0.21)},anchor=west,legend cell align=left, align=left, draw=white!15!black}
]
                
\addplot [color=CN, line width=2.0pt]
  table[row sep=crcr]{%
0.2000  0.003840648783117\\
0.100  7.077096081645068e-04\\
0.050  1.572935067027902e-04\\
0.025  3.804466860706695e-05\\
};
\addlegendentry{\textbf{CN:} $||c(T)-c_h(T)||_\varepsilon$}

\addplot [color=CN, dashed,line width=2.0pt]
  table[row sep=crcr]{%
0.2000  0.004889584476105\\
0.100  9.887466196786465e-04\\
0.050  2.289501896475613e-04\\
0.025  5.591710693127719e-05\\
};
\addlegendentry{\textbf{CN:} $||q(T)-q_h(T)||_\varepsilon$}
 
\addplot [color=EI, line width=2.0pt]
  table[row sep=crcr]{%
0.2000  0.026259597346244\\
0.100  0.013500751325132\\
0.050  0.006847585818458\\
0.025  0.003448897172470\\
};
\addlegendentry{\textbf{BE:} $||c(T)-c_h(T)||_\varepsilon$}

    \addplot [color=EI, dashed, line width=2.0pt]
  table[row sep=crcr]{%
0.2000 0.035618061402185 \\
0.100  0.018277916696053\\
0.050  0.009262610333967\\
0.025  0.004666654955709\\
};
\addlegendentry{\textbf{BE:} $||q(T)-q_h(T)||_\varepsilon$}

\node[right, align=left, text=black, font=\footnotesize]
at (axis cs:0.04,3e-3) {$1$};

\addplot [color=black, line width=1.5pt]
  table[row sep=crcr]{%
0.039   4.5e-3\\
0.03   3e-3 \\
0.039   3e-3\\
0.039   4.50e-3\\
};

\node[right, align=left, text=black, font=\footnotesize]
at (axis cs:0.04,4e-5) {$2$};

\addplot [color=black, line width=1.5pt]
  table[row sep=crcr]{%
0.039   7.0e-5\\
0.030   4e-5\\
0.039   4e-5\\
0.039   7.0e-5\\
};
\end{axis}
\end{tikzpicture}}
                \caption{Test case 1: computed errors with respect to $\Delta t$.}
            \label{fig:test1_errordt}
    \end{subfigure}%
    \caption{Test case 1: computed errors and convergence rates with respect to $p$ (left) and $\Delta t$ (right). The results on the left panel have been obtained with the Crank-Nicolson (CN) time integration scheme; the results on the right panel have been obtained with both Backward Euler (BE) and CN schemes.}
\end{figure}
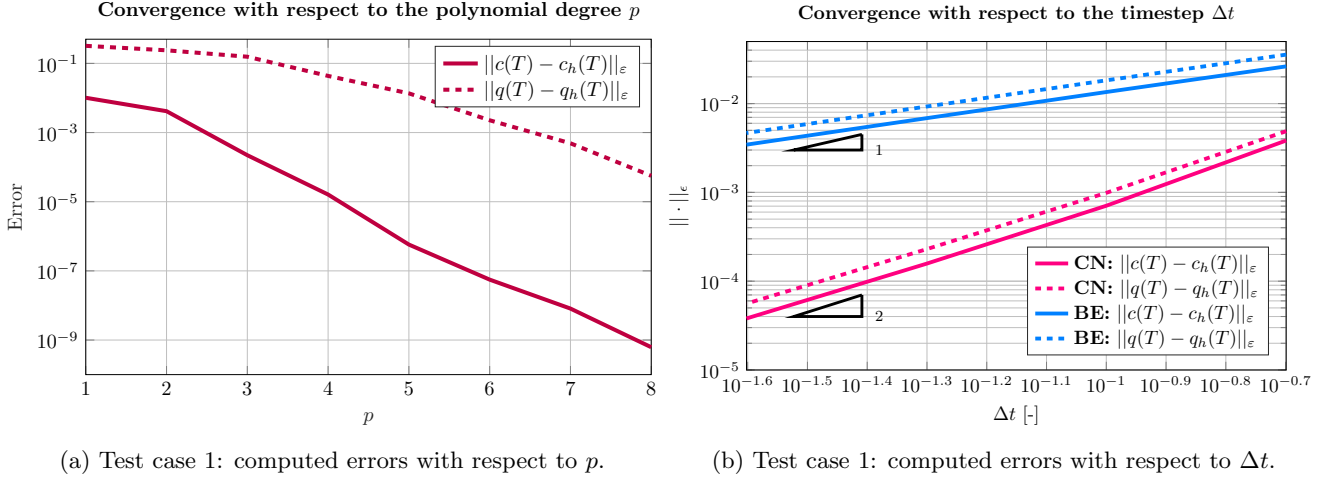
We establish a predefined exact solution as:
\begin{equation*}
     c(x,y,t) = (\cos(\pi x)+\cos(\pi y))\cos(t), \qquad
     q(x,y,t) = (\cos(4\pi x)\cos(4\pi y)+2)e^{-t}.
\end{equation*}
The forcing terms and the Dirichlet boundary conditions are chosen accordingly. In Figure~\ref{fig:test1_errorsh} we report the numerical results of the convergence test performed with respect to the mesh size parameter $h$. We use a time discretization step $\Delta t = 10^{-5}$ and fix a final time instant $ T= 10^{-3}$. We collect the energy norm of the error as defined in Equation~\eqref{eq:energynorm} for each solution and repeat the test keeping fixed the polynomial degree of approximation $p = 1,2,3,4$, while progressively refining the mesh ($N_{el}=30,100,300,1000$, where $N_{el}$ denotes the number of polygons of the domain partition). We can observe that the optimal rate of convergence is always achieved as expected from theory.
\par
In Figure~\ref{fig:test1_errorsp} we report the energy norm of the error with respect the polynomial degree $p$. We repeat the numerical tests keeping fixed the polygonal mesh characterized by $N_{el}=30$ and the time discretization $\Delta t=10^{-5}$ while considering an increasing polynomial order $p=1,...,8$. Even if the dependency of the error from $p$ has not been covered by the theoretical analysis we numerically demonstrated that exponential convergence is obtained.
\par
Finally, a convergence test with respect to the time discretization step $\Delta t$ has been reported in Figure~\ref{fig:test1_errordt}. The test set the most refined polygonal mesh with $N_{el}= 1000$ and $p=6$, and we registered the error in energy norm obtained for different time steps $\Delta t = 0.2, 0.1, 0.05, 0.025 $. We can observe that the Backward Euler method (characterized by setting $\theta=1$) exhibits theoretical linear convergence with respect to $\Delta t$, whereas the error decays with a second-order rate for Crank-Nicolson method ($\theta=0.5$).

\subsection{Test 2: simulation of a travelling wave}
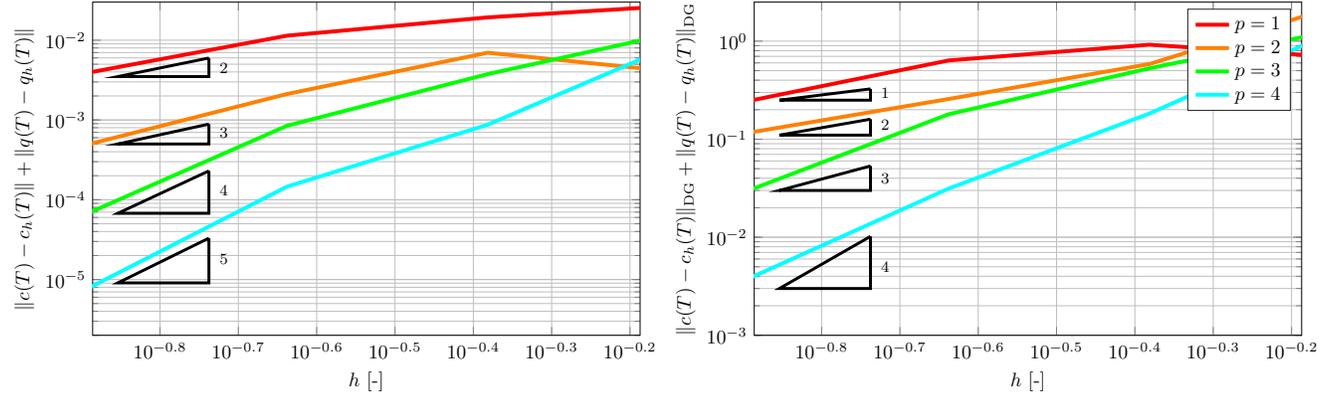
\begin{figure}[t!]
    \begin{subfigure}[b]{0.5\textwidth}
            \resizebox{\textwidth}{!}{\definecolor{mycolor2}{rgb}{0.00000,1.00000,1.00000}%
\pgfplotsset{
  log x ticks with fixed point/.style={
      xticklabel={
        \pgfkeys{/pgf/fpu=true}
        \pgfmathparse{exp(\tick)}%
        \pgfmathprintnumber[fixed  zerofill, precision=2]{\pgfmathresult}
        \pgfkeys{/pgf/fpu=false}
      }
  }
}
\begin{tikzpicture}

\begin{axis}[%
width=3.875in,
height=2.36in,
at={(2.6in,1.099in)},
scale only axis,
xmode=log,
xmin=0.13,
xmax=0.65,
xminorticks=true,
xlabel = {$h$ [-]},
ylabel = {$\|c(T)-c_h(T)\|+\|q(T)-q_h(T)\|$},
ymode=log,
ymin=2e-6,
ymax=0.030,
yminorticks=true,
axis background/.style={fill=white},
title style={font=\bfseries},
xmajorgrids,
xminorgrids,
ymajorgrids,
yminorgrids,
legend style={legend cell align=left, align=left, draw=white!15!black}
]
              
\addplot [color=red, line width=2.0pt]
  table[row sep=crcr]{%
    0.735612457097084  0.027365807350492 \\
    0.415320484644076  0.019315050542246 \\
    0.230266920810526  0.011397411488666 \\
    0.128860161084044  0.003946793721678 \\
};

\addplot [color=orange, line width=2.0pt]
  table[row sep=crcr]{%
    0.735612457097084   0.003946793721678 \\
    0.415320484644076	0.006963905528060 \\
    0.230266920810526	0.002118911721644 \\
    0.128860161084044   4.99136813488e-04 \\
};

\addplot [color=green, line width=2.0pt] 
  table[row sep=crcr]{%
    0.735612457097084	0.012915652189790 \\
    0.415320484644076   0.003758747129800 \\
    0.230266920810526	8.47881425602e-04 \\
    0.128860161084044   6.91446446962e-05 \\
};

\addplot [color=mycolor2, line width=2.0pt]
  table[row sep=crcr]{%
    0.735612457097084   0.009594308601130 \\
    0.415320484644076	8.72477374670e-04 \\
    0.230266920810526	1.46220273867e-04 \\
    0.128860161084044	7.93927403266e-06 \\
};

\node[right, align=left, text=black, font=\footnotesize]
at (axis cs:0.185,4.5e-3) {$2$};

\addplot [color=black, line width=1.5pt]
  table[row sep=crcr]{%
    0.183   5.98e-3\\
    0.140   3.50e-3\\
    0.183   3.50e-3\\
    0.183   5.98e-3\\
};

\node[right, align=left, text=black, font=\footnotesize]
at (axis cs:0.185,7e-4) {$3$};

\addplot [color=black, line width=1.5pt]
  table[row sep=crcr]{%
    0.183   8.88e-4\\
    0.140   5.00e-4\\
    0.183   5.00e-4\\
    0.183   8.88e-4\\
};

\node[right, align=left, text=black, font=\footnotesize]
at (axis cs:0.185,1.3e-4) {$4$};

\addplot [color=black, line width=1.5pt]
  table[row sep=crcr]{%
    0.183   2.30e-4\\
    0.140   6.76e-5\\
    0.183   6.76e-5\\
    0.183   2.30e-4\\
};

\node[right, align=left, text=black, font=\footnotesize]
at (axis cs:0.185,1.8e-5) {$5$};

\addplot [color=black, line width=1.5pt]
  table[row sep=crcr]{%
    0.183   3.300e-5\\
    0.140   9.075e-6\\
    0.183   9.075e-6\\
    0.183   3.300e-5\\
};

\end{axis}
\end{tikzpicture}%
}
    \end{subfigure}%
    \begin{subfigure}[b]{0.5\textwidth}
            \resizebox{\textwidth}{!}{\definecolor{mycolor2}{rgb}{0.00000,1.00000,1.00000}%
\pgfplotsset{
  log x ticks with fixed point/.style={
      xticklabel={
        \pgfkeys{/pgf/fpu=true}
        \pgfmathparse{exp(\tick)}%
        \pgfmathprintnumber[fixed  zerofill, precision=2]{\pgfmathresult}
        \pgfkeys{/pgf/fpu=false}
      }
  }
}
\begin{tikzpicture}

\begin{axis}[%
width=3.875in,
height=2.36in,
at={(2.6in,1.099in)},
scale only axis,
xmode=log,
xmin=0.13,
xmax=0.65,
xminorticks=true,
xlabel = {$h$ [-]},
ylabel = {$\|c(T)-c_h(T)\|_\mathrm{DG}+\|q(T)-q_h(T)\|_\mathrm{DG}$},
ymode=log,
ymin=1e-3,
ymax=2.5,
yminorticks=true,
axis background/.style={fill=white},
title style={font=\bfseries},
xmajorgrids,
xminorgrids,
ymajorgrids,
yminorgrids,
legend style={legend cell align=left, align=left, draw=white!15!black}
]
              
\addplot [color=red, line width=2.0pt]
  table[row sep=crcr]{%
    0.735612457097084  0.681846402049076 \\
    0.415320484644076  0.919658934954176 \\
    0.230266920810526  0.635897816859510 \\
    0.128860161084044  0.248485941004050 \\
};
\addlegendentry{$p=1$}

\addplot [color=orange, line width=2.0pt]
  table[row sep=crcr]{%
    0.735612457097084   2.398892910347856 \\
    0.415320484644076	0.583678193872960 \\
    0.230266920810526	0.257119381397100 \\
    0.128860161084044   0.117579355827332 \\
};
\addlegendentry{$p=2$}

\addplot [color=green, line width=2.0pt] 
  table[row sep=crcr]{%
    0.735612457097084	1.340036044777102\\
    0.415320484644076   0.528128594138934\\
    0.230266920810526	0.179717340471348\\
    0.128860161084044   0.030696536821256\\
};
\addlegendentry{$p=3$}

\addplot [color=mycolor2, line width=2.0pt]
  table[row sep=crcr]{%
    0.735612457097084   1.385785299981486\\
    0.415320484644076	0.182379188357632\\
    0.230266920810526	0.031364825624356\\
    0.128860161084044	0.003894161405588\\
};
\addlegendentry{$p=4$}

\node[right, align=left, text=black, font=\footnotesize]
at (axis cs:0.185,0.30) {$1$};

\addplot [color=black, line width=1.5pt]
  table[row sep=crcr]{%
    0.183   0.326\\
    0.140   0.250\\
    0.183   0.250\\
    0.183   0.326\\
};

\node[right, align=left, text=black, font=\footnotesize]
at (axis cs:0.185,0.14) {$2$};

\addplot [color=black, line width=1.5pt]
  table[row sep=crcr]{%
    0.183   0.160\\
    0.140   0.110\\
    0.183   0.110\\
    0.183   0.160\\
};

\node[right, align=left, text=black, font=\footnotesize]
at (axis cs:0.185,4e-2) {$3$};

\addplot [color=black, line width=1.5pt]
  table[row sep=crcr]{%
    0.183   5.33e-2\\
    0.140   3.00e-2\\
    0.183   3.00e-2\\
    0.183   5.33e-2\\
};

\node[right, align=left, text=black, font=\footnotesize]
at (axis cs:0.185,5e-3) {$4$};

\addplot [color=black, line width=1.5pt]
  table[row sep=crcr]{%
    0.183   1.02e-2\\
    0.140   3.00e-3\\
    0.183   3.00e-3\\
    0.183   1.02e-2\\
};

\end{axis}
\end{tikzpicture}%
}
    \end{subfigure}%
    \caption{Test case 2: computed errors at final time $T=10^{-3}$ and computed convergence rates versus $h$ for different choices of the polynomial degree $p=1,2,3,4$ and for the solutions $c_h$ (left) and $q_h$ (right).}
    \label{fig:test2_errors}
\end{figure}
\begin{table}[t!]

    \centering
    \begin{tabular}{c|c|C|C|c|C|C}
    \hline
    \textbf{Method}
    & \multicolumn{3}{c|}{\textbf{$p=1$}} 
    & \multicolumn{3}{c}{\textbf{$p=2$}}
    \\    \hline
    \textbf{$h$} & \textbf{DOFs}
    & $T_1 = 5$    & $T_2 = 10$  & \textbf{DOFs}
    & $T_1 = 5$    & $T_2 = 10$
    \\ \hline
    $0.736$ & $90$
    & $4.62 \times 10^{-2}$ & $4.62 \times 10^{-2}$ & $180$
    & $2.74 \times 10^{-2}$ & $2.67 \times 10^{-2}$
    \\ \hline
    $0.415$ & $300$
    & $2.83 \times 10^{-2}$ & $2.74 \times 10^{-2}$ & $600$
    & $8.79 \times 10^{-3}$ & $9.38 \times 10^{-3}$
    \\ \hline    
    $0.230$ & $900$
    & $1.55 \times 10^{-2}$ & $1.76 \times 10^{-2}$ & $1800$
    & $2.33 \times 10^{-3}$ & $2.78 \times 10^{-3}$
    \\ \hline
    $0.129$ & $3000$
    & $6.24 \times 10^{-3}$ & $5.65 \times 10^{-3}$ & $6000$
    & $4.63 \times 10^{-4}$ & $4.39 \times 10^{-4}$
    \\ \hline
    \end{tabular}
    \\[10pt]
    \begin{tabular}{c|c|C|C|c|C|C}
    \hline
    \textbf{Method}
    & \multicolumn{3}{c|}{\textbf{$p=3$}} 
    & \multicolumn{3}{c}{\textbf{$p=4$}}
    \\    \hline
    \textbf{$h$} & \textbf{DOFs}
    & $T_1 = 5$    & $T_2 = 10$  & \textbf{DOFs}
    & $T_1 = 5$    & $T_2 = 10$
    \\ \hline
    $0.736$ & $300$
    & $8.09 \times 10^{-3}$ & $9.07 \times 10^{-3}$ & $450$
    & $8.19 \times 10^{-3}$ & $7.78 \times 10^{-3}$
    \\ \hline
    $0.415$ & $1000$
    & $3.22 \times 10^{-3}$ & $2.77 \times 10^{-3}$ & $1500$
    & $1.21 \times 10^{-3}$ & $1.19 \times 10^{-3}$
    \\ \hline    
    $0.230$ & $3000$
    & $5.29 \times 10^{-4}$ & $5.31 \times 10^{-4}$ & $4500$
    & $1.29 \times 10^{-4}$ & $2.05 \times 10^{-4}$
    \\ \hline
    $0.129$ & $10000$
    & $4.99 \times 10^{-5}$ & $4.77 \times 10^{-5}$ & $15000$
    & $5.71 \times 10^{-6}$ & $5.29 \times 10^{-6}$
    \\ \hline
    \end{tabular}
    \caption{Computed errors with $L^2$-norm of the error of $\|q(T)-q_h(T)\|$ at different final time instants (errors for $c_h$ are analogous).}
    \label{table:Test2}
\end{table}
The objective of this test is to evaluate the method's ability to accurately replicate the dynamics of a travelling-wave solution as it drives the system towards the unstable equilibrium state. To accomplish this we consider a rectangular domain $\Omega=(0,5) \times (0,1)$ partitioned via Polymesher \cite{Polymesher}. We impose as the exact solution a couple of travelling-wave functions characterized by a velocity of propagation $\bar{v}=0.1$:
\begin{align*}
       c(x,y,t) = \frac{\arctan( 3 \pi ( x-\bar{v} t - 1))}{\pi} +\frac{1}{2}, \qquad
       q(x,y,t) = \frac{\arctan(-3 \pi ( x-\bar{v} t - 1))}{\pi} +\frac{1}{2},
\end{align*}
and we derive the forcing terms and Dirichlet conditions accordingly. We adopt the parameter values shown in Table~\ref{tab:test1_param}.
\par
A convergence analysis has been performed as in the previous section, in particular, we collect in Figure~\ref{fig:test2_errors} the $L^2$-norm and the DG-norm of the error with respect to the mesh size parameter $h$. As expected we can observe that the DG-norm of the error decreases as $h^p$, while the $L^2$-norm decreases as $h^{p+1}$. 
\par
\begin{figure}[t!]
	\centering
	\begin{subfigure}[b]{\textwidth}
        {\includegraphics[width=\textwidth]{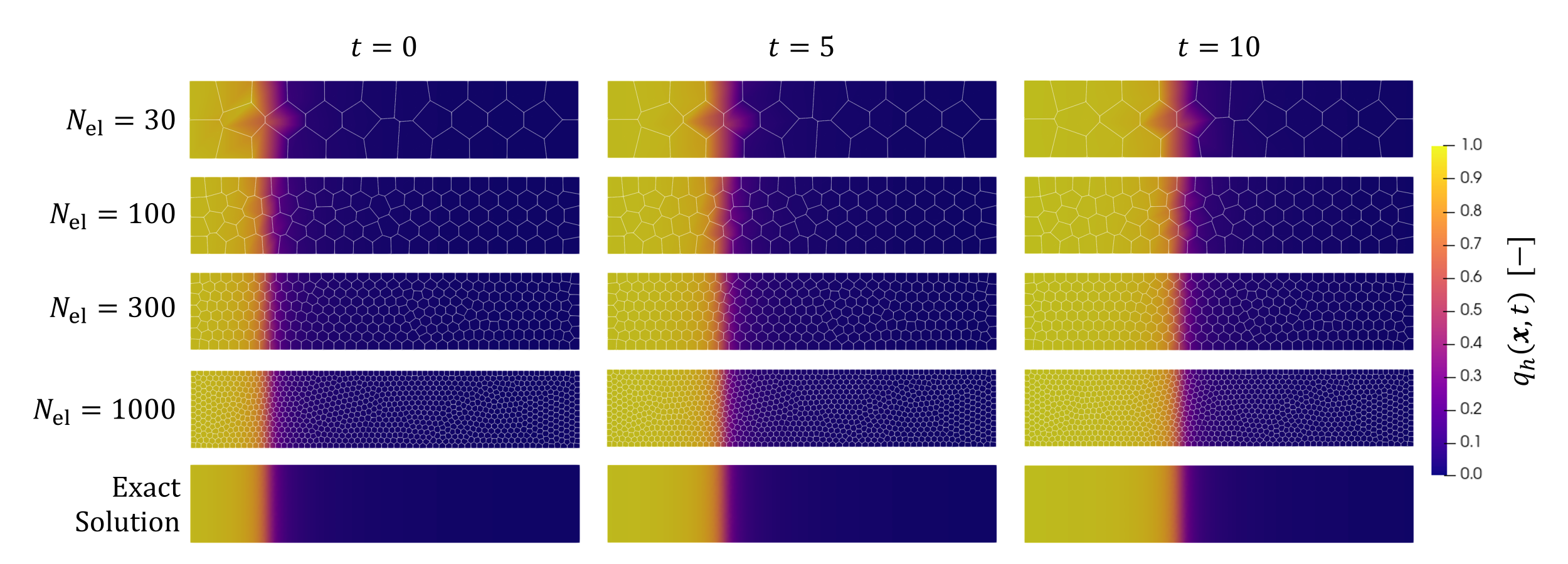}}
        \caption{Wavefront propagation of $q$.}
    \end{subfigure}%
    \\
	\begin{subfigure}[b]{\textwidth}
        {\includegraphics[width=\textwidth]{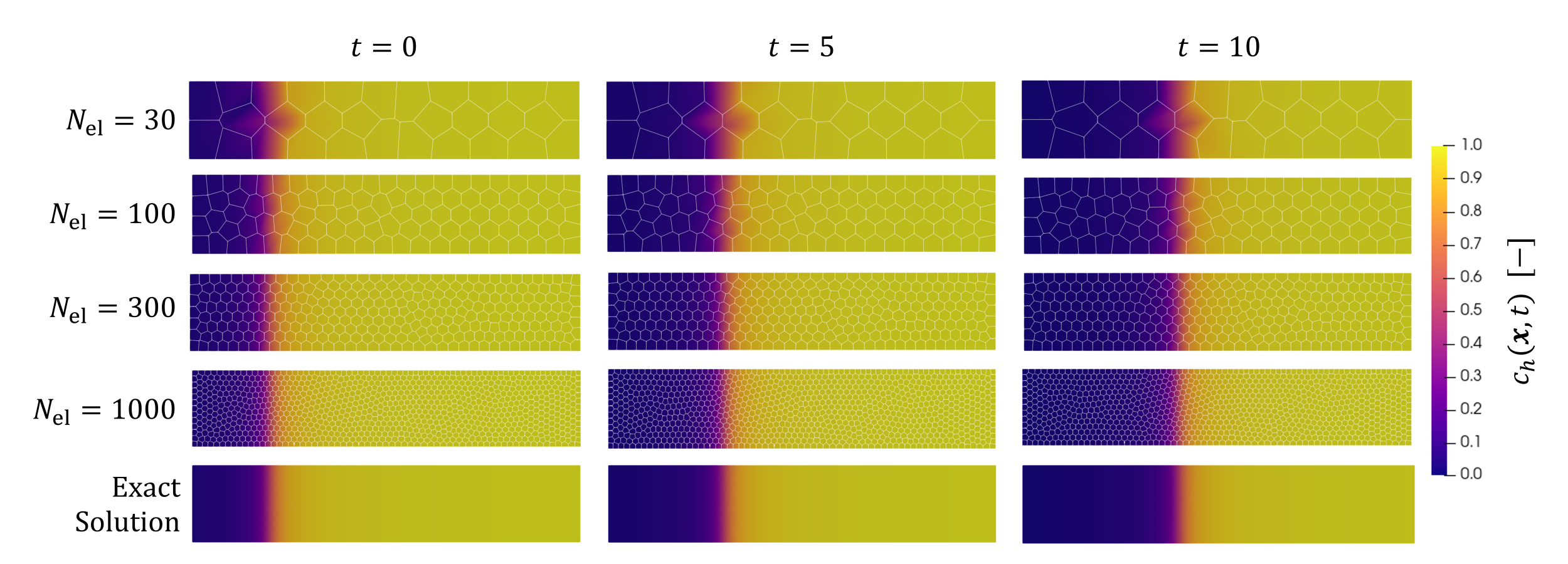}}
        \caption{Wavefront propagation of $c$.}
    \end{subfigure}%
    \caption{Test case 2: wavefront propagation approximated with $p = 2$.}
	\label{fig:wavefront}
\end{figure}  
We repeat the test for two possible final time instants $T_1=5$ and $T_2=10$. These results aimed at assessing the impact of mesh refinement on the accuracy of the solution where the error is measured in the $L^2$-norm for various polynomial orders, specifically $p=1,2,3,4$. The results of our investigations are reported in Table~\ref{table:Test2}. Notably, the $L^2$-norms of errors at different final time instants are comparable, which supports the fact that the computed solution is stable and correctly approximated. Moreover, with this test, we can highlight the capability of the higher-order methods to accurately replicate the wave behaviour of the solution even with coarser mesh configurations \cite{antonietti_mathematical_2022}. To illustrate this point, consider the error obtained with $N_{el}=1000$ and $p=1$ compared to that with $N_{el}=300$ and $p=3$. Even if both setups have the same number of degrees of freedom (3000), the solution approximated with $p=3$ exhibited an error of $5.29 \times 10^{-4}$, which is approximately one order of magnitude smaller than the error in the case $p=1$ ($6.24 \times 10^{-3}$). This discrepancy becomes even more evident when comparing the accuracy of the solution achieved with $p=4$ and $N_{el}=100$ to that obtained with $p=2$ and $N_{el}=300$. Despite having fewer degrees of freedom, the solution computed with $p=4$ on a coarser mesh is more accurate than the one obtained with $p=2$.  
\par
Finally, in Figure~\ref{fig:wavefront}, we visualize the evolution of the wavefront for the case of $p=2$ at the two chosen final times, for different levels of mesh refinement. We stress the fact that a sufficiently refined grid is essential to capture the rapid variations characterizing the wave. However, even with $p=2$, we can confidently conclude that the method produces a numerical solution that accurately reproduces the wave, together with its propagation velocity.
\subsection{Test case 3: diffusion processes on agglomerated brain meshes}
In this section, we consider a detailed mesh of the brain obtained by segmentation of a structural MRI dataset sourced from the OASIS-3 database \cite{OASIS3}, as described in \cite{corti_discontinuous_2023}. Subsequently, the initial fine triangulated structure, comprising $43 402$ triangles (of Figure~\ref{fig:BrainMeshTri}), is agglomerated into a polygonal mesh consisting of $534$ elements, of Figure~\ref{fig:BrainMeshAgg}. This transformation gives us the possibility to take full advantage of the versatility of the PolyDG method reducing computational overhead during the simulation.
\begin{figure}[t]
    \centering
    \begin{subfigure}[b]{0.4\textwidth}
        \includegraphics[width=\textwidth]{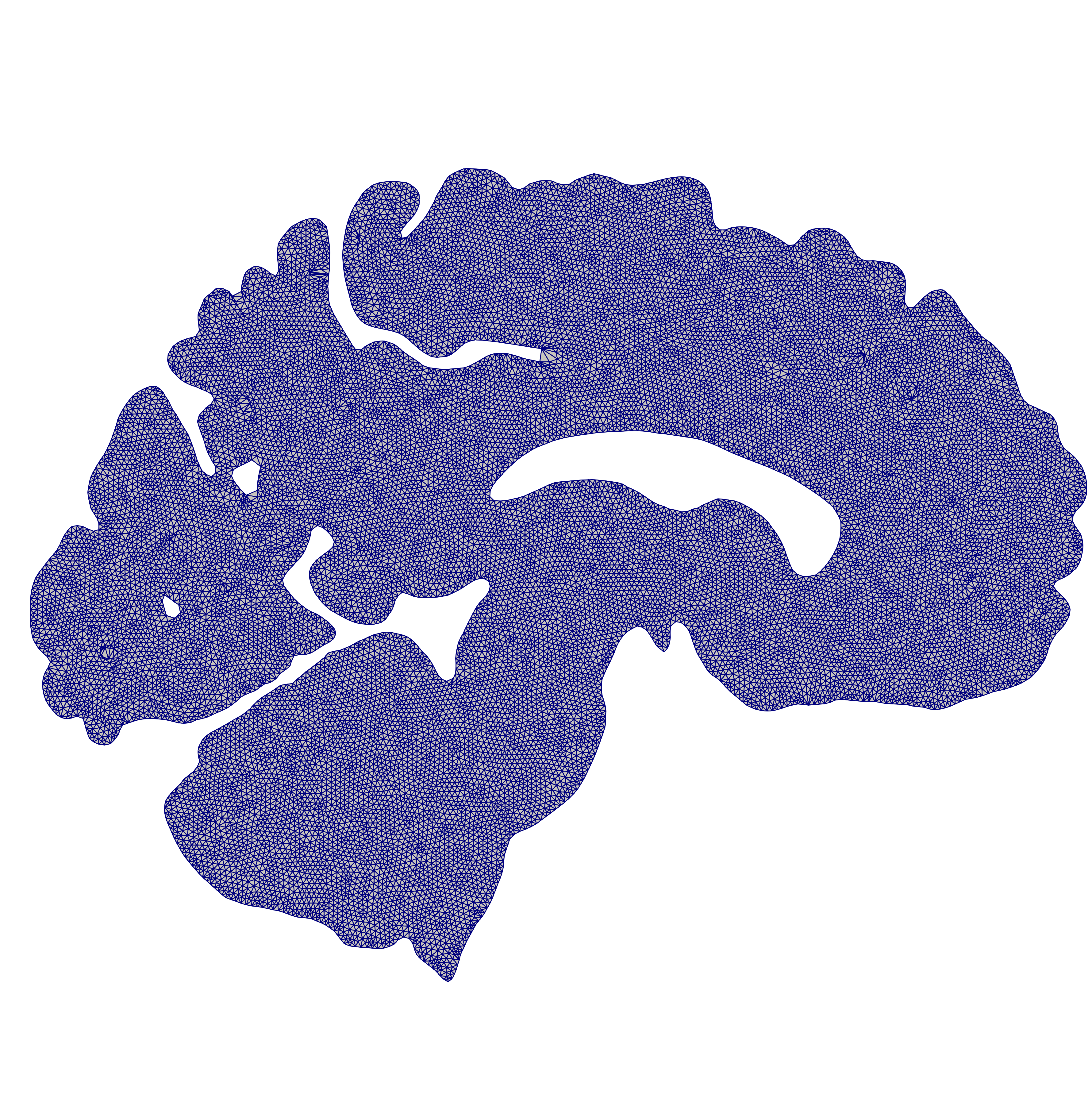}
        \caption{Triangular mesh}
        \label{fig:BrainMeshTri}
    \end{subfigure}
    \begin{subfigure}[b]{0.4\textwidth}
        \includegraphics[width=\textwidth]{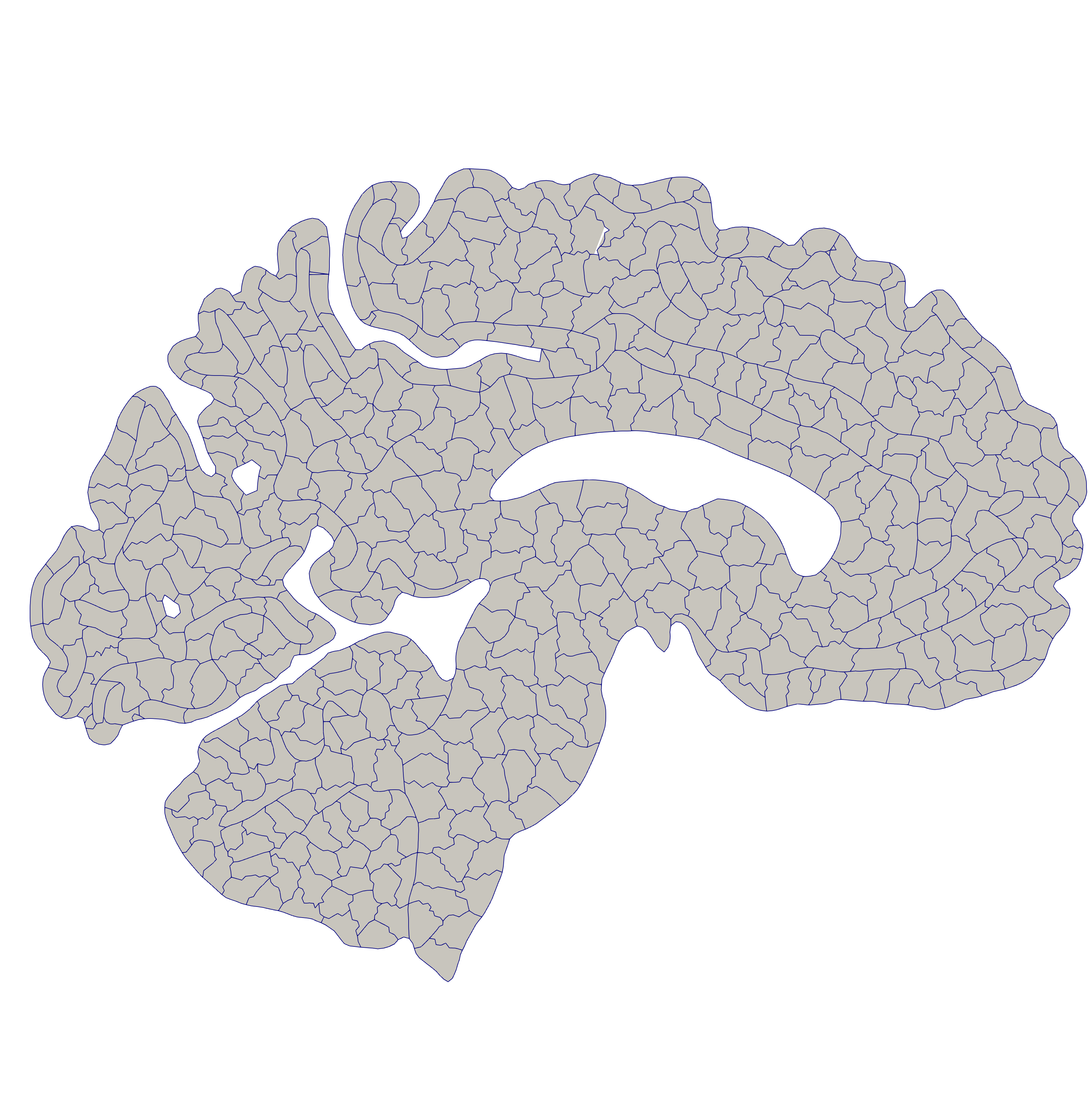}
        \caption{Agglomerated polygonal mesh}
        \label{fig:BrainMeshAgg}
    \end{subfigure}
   \caption{Initial (a) and agglomerated (b) meshes of a two-dimensional brain section reconstructed from MRI image.}
   \label{fig:BrainMesh}
\end{figure}

\begin{table}[t]
	\centering
	\begin{tabular}{|c|r l|c|r l|}
	\hline
	\textbf{Parameter} & \multicolumn{2}{c|}{\textbf{Value}} &	\textbf{Parameter} & \multicolumn{2}{c|}{\textbf{Value}} \\ 
		\hline 
		 $d_\mathrm{ext}$     &  $10^{-6}$    & $[\mathrm{mm^2/years}]$ 
		& $d_\mathrm{axn}$     &  $0.00$       & $[\mathrm{mm^2/years}]$ \\ 
		\hline 
		 $k_0$                 &  $0.75$       & $[\mathrm{1/years}]$ 
		& $k_{12}$              &  $1.00$       & $[\mathrm{1/years}]$ \\ 
		\hline 
		 $k_1$                 &  $1.00$       & $[\mathrm{1/years}]$ 
		& $\widetilde{k}_1$     &  $0.60$       & $[\mathrm{1/years}]$ \\ 
		\hline 
	\end{tabular}
	\caption{Physical parameter values used in Test case 3 \cite{thompsonProteinproteinInteractionsNeurodegenerative2020a}.}
	\label{tab:test3_param}
\end{table}

\par
The goal of this test is to test the capability of the method to correctly reproduce the isotropic diffusion phenomena on a polygonal agglomerated grid of a brain section. The initial conditions we impose are the following:
\begin{equation*}
    c_0(x,y) = 0.75 \qquad q_0(x,y) = 0.25 e^{-50(x-0.02)^2-50y^2}.
\end{equation*}
\begin{figure}[t]
    \centering
    \includegraphics[width=\textwidth]{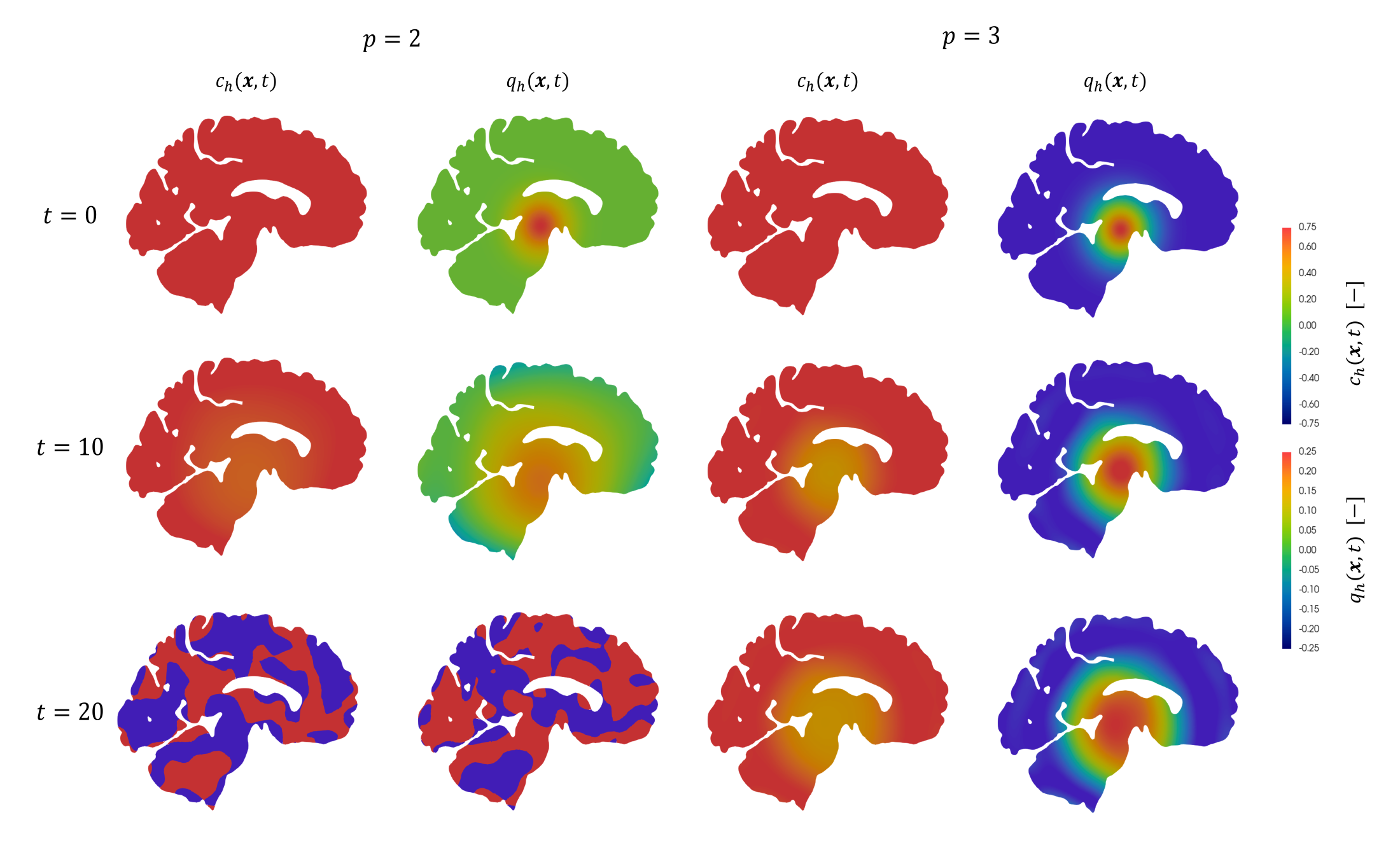}
   \caption{Test case 3: Numerical solution obtained by using $p=2$ (left) and $p = 3$ (right).}
   \label{fig:solution_brain_waves}
\end{figure}
The values of the model parameters are reported in Table~\ref{tab:test3_param}. Analyzing the model we expect to observe the system evolve towards the pathological stable equilibrium point presented in Section \ref{MathematicalModel}:
\begin{equation}
    \label{eq_Test3}
    (c_2,q_2)=\left(\frac{\tilde{k}_1}{k_{12}},\frac{k_0}{\tilde{k}_1}-\frac{k_1}{k_{12}}\right)= (0.6,0.25).
\end{equation}
Moreover, it is possible to compute analytically the minimum speed of the travelling wave solution transitioning the system from the initial solutions toward the pathological state: 
\begin{equation*}
\bar{v}= \displaystyle 2 \sqrt{d_{ext} \left(k_{12}\frac{k_0}{k_1}-\tilde{k}_1\right)}\simeq 6 \times 10^{-5}.
\end{equation*}
The problem is coupled with homogeneous Neumann boundary conditions.
\par
In Figure~\ref{fig:solution_brain_waves}, we report screenshots of the approximated solution starting from the initial solution up to the final time $T=20$, $\Delta t = 10^{-2}$ and $p=2, 3$. Considering the case $p=2$, we detect an unstable approximated solution, even when employing a reduced time step. Our observations reveal that the employed method lacks positive preservation characteristics. As a consequence, when the approximated solution manifests negative values, the computational scheme converges towards an oscillatory solution with no physical significance. This behavior underscores the need for improved positivity-preserving algorithms in order to maintain the integrity of the numerical solutions in the presence of negative values as in the case of FK model \cite{corti_positivity_preserving}. Nevertheless, stability can be reinstated through the computation of the approximation at a higher order, specifically with $p=3$, as illustrated in Figures~\ref{fig:solution_brain_waves}. In a reaction-dominated regime, higher-order methods can effectively capture the travelling wave solution, resulting in the recovery of a stable solution transitioning the system towards the stable equilibrium. This behaviour is coherent with the recent results in wave-front simulations studies \cite{antonietti_mathematical_2022}. This numerical test confirms the capacity of the proposed method of representing the wavefronts in complex geometries, like in brain sections, where the mesh are constructed as agglomeration of detailed triangular meshes, if using sufficiently high order polynomial degree in the approximation. This is a fundamental property to guarantee the quality of the solution in the medical applications.
\section{Simulation of $\alpha$-synuclein spreading in a two dimensional brain section}
\label{BrainApp}
\begin{table}[t]
	\centering
	\begin{tabular}{|c|r l|c|r l|}
	\hline
	\textbf{Parameter} & \multicolumn{2}{c|}{\textbf{Value}} &	\textbf{Parameter} & \multicolumn{2}{c|}{\textbf{Value}} \\ 
		\hline 
		 $d_\mathrm{ext}$     &  $8.00$        & $[\mathrm{mm^2/years}]$ 
		& $d_\mathrm{axn}$     &  $80.00$       & $[\mathrm{mm^2/years}]$ \\ 
		\hline 
		 $k_0$                 &  $0.60$       & $[\mathrm{1/years}]$ 
		& $k_{12}$              &  $1.00$       & $[\mathrm{1/years}]$ \\ 
		\hline 
		 $k_1$                 &  $0.50$       & $[\mathrm{1/years}]$ 
		& $\widetilde{k}_1$     &  $0.30$       & $[\mathrm{1/years}]$ \\ 
		\hline 
	\end{tabular}
	\caption{Physical parameter values used in Test 4 \cite{thompsonProteinproteinInteractionsNeurodegenerative2020a}.}
	\label{tab:test4_param}
\end{table}
\begin{figure}[t]
    \centering
    \includegraphics[width=\textwidth]{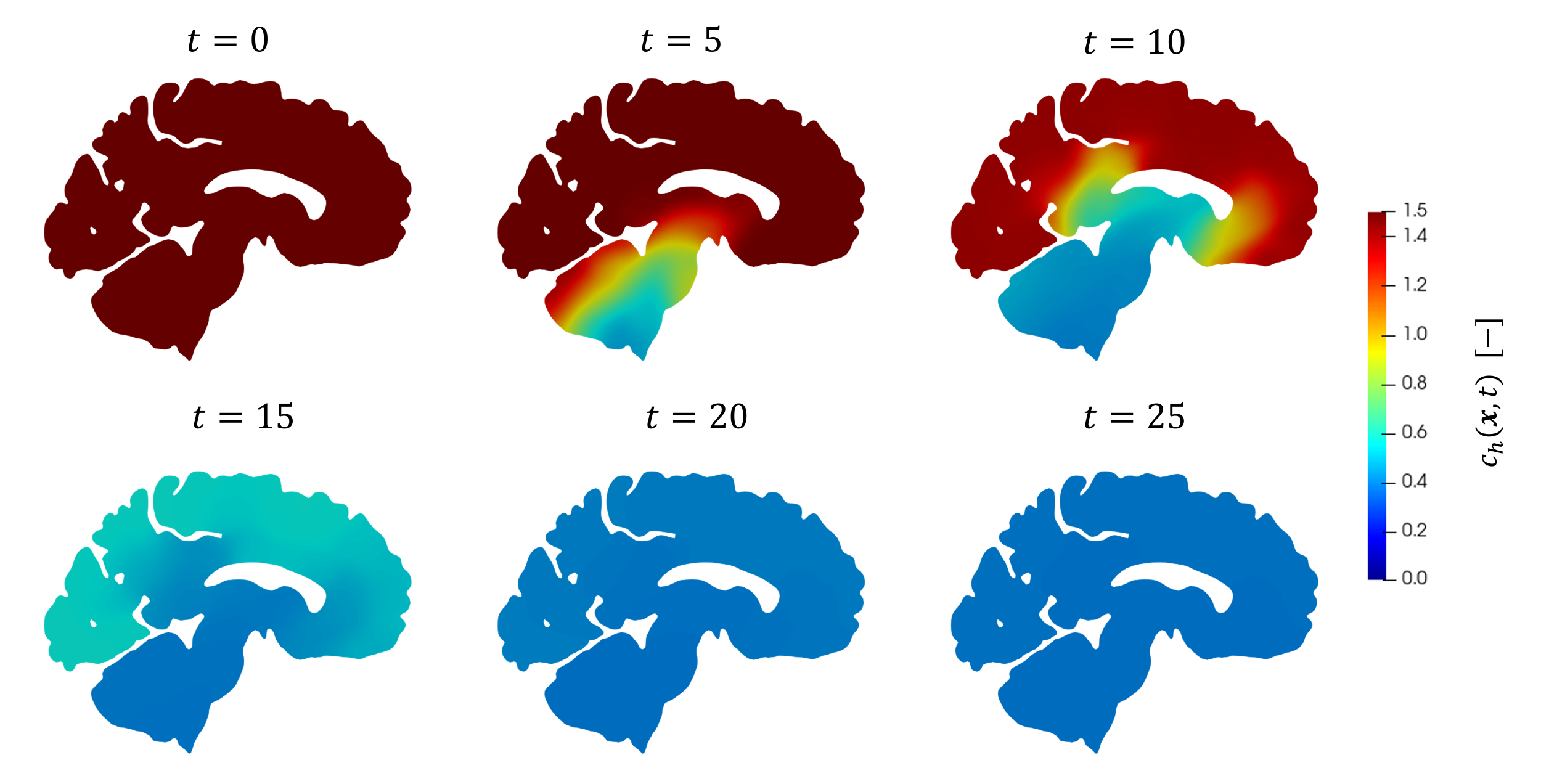}
   \caption{The initial configuration corresponds to the unstable healthy state, with a protein concentration $c=1.2$, uniformily distributed all over the brain. The presence of misfolded proteins trigger the system causing the decay of the concentration.}
   \label{fig:solution_brain_c}
\end{figure}
This Section is devoted to address the problem of simulating the spread of $\alpha$-synuclein protein into a two dimensional section of the human brain, associated to the Parkinson's disease \cite{spillantini_-synuclein_1997}. Coherently to the modelling of this phenomenon, we impose null flux across the boundary, setting $\partial \Omega = \Gamma_\mathrm{N}$. 
\par
In order to model the diffusion of $\alpha$-synuclein protein within the previously delineated brain region, we need to establish appropriate parameter configurations. Measuring such mean rate of production, destruction and conversion to provide best-fit coefficients to deterministic model constitutes a developing field in current research \cite{li_quantifying_2005,masel_quantifying_1999,iljina_kinetic_2016,perrino_quantitative_2019}.
\par
In \cite{corti_discontinuous_2023}, an exploration of the phenomenon under investigation in this section was conducted through the implementation of the FK model. This model, derived from the heterodimer equations, as elucidated in a prior study \cite{weickenmeierPhysicsbasedModelExplains2019}, offers a simplified representation, focusing on the temporal evolution of the relative concentration of misfolded protein only. The fundamental parameter of interest in this context is the reaction coefficient denoted as $\alpha$. This parameter holds a phenomenological significance and exhibits a direct connection with the coefficients of the heterodimer model, as established by the following relationship:  
\begin{equation}
    \label{eq:alpha}
    \alpha = k_{12}\frac{k_0}{k_1}- \tilde{k}_{1}.
\end{equation}
Specifically, a significant value for this parameter, one that facilitates the faithful replication of the spreading pattern and propagation velocity of $\alpha$-synuclein in the brain, may be established at $\alpha=0.9/\mathrm{year}$ \cite{weickenmeierPhysicsbasedModelExplains2019}. Consequently, an appropriate parameter configuration, one that adheres to the relationship delineated in Equation~\eqref{eq:alpha} and aligns with the considerations in Section \ref{MathematicalModel}. 
\par
\par
Another important aspect of the model is the construction of the anisotropic part of the diffusion tensor in Equation~\eqref{eq:DiffusionTensor}. To fulfill this purpose, the axonal directions $\bar{\boldsymbol{a}}(\boldsymbol{x})$ are derived from diffusion-weighted imaging (DWI) \cite{corti_discontinuous_2023}. The adopted values are reported in Table~\ref{tab:test4_param}. We trigger the system with an initial seeding of the misfolded version of $\alpha$-synuclein protein concentrated in the dorsal motor nucleus \cite{dickson_neuropathology_2018}, while we set an initial condition for the healthy protein concentration uniformly equal to the unstable equilibrium point $1.2$. We remark that, using this setting and considering the stability estimate in equation \eqref{StabilityEstimate_HomofeneousF}, we can obtain the following order of the final time $T$:
\begin{equation*}
    T = \left(\dfrac{\tilde{\mu}2^{d-2}\mu}{C_{\mathrm{G}_d}^3 K_{12}^2 k_0^{2(d-1)} |\Omega|^{2(d-1)}}\right)^{1/d} = O(10^8\,\mathrm{s}).
\end{equation*}
We fix the polynomial order $p=4$ in every element of the discretization, Crank-Nicolson method for the time discretization with time step $\Delta t = 0.01$ and simulate up to $T=25$ years. We have opted for a second-order numerical method in order to ensure the accurate approximation of the wavefront over an extended temporal simulation interval. 
\par
\begin{figure}[t]
    \centering
    \includegraphics[width=\textwidth]{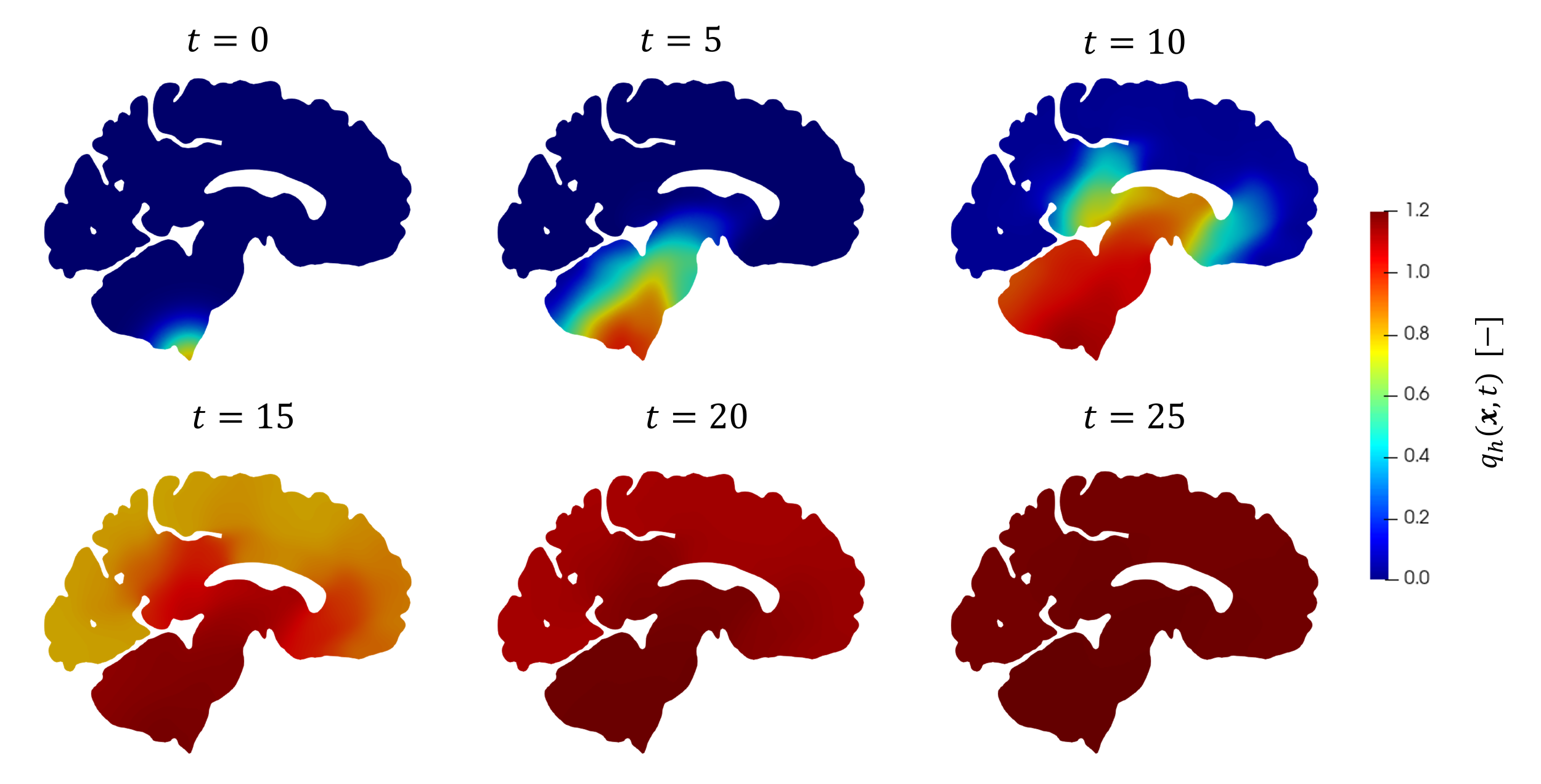}
   \caption{Misfolded protein $q_h$, initially concentrate into the dorsal motor nucleus, spreads through the entire brain section.}
   \label{fig:solution_brain_q}
\end{figure}
\par
In this framework we expect the system to develop from the initial state towards the stable equilibrium point $(c_h,q_h)=(0.3,1.5)$. In Figures \ref{fig:solution_brain_c} and \ref{fig:solution_brain_q}, we report the approximated solution at different time instants of the simulation. The solutions have effectively developed toward a stable pathological state. We compare the misfolded evolution in Figure \ref{fig:solution_brain_q} with the solution obtained in other works in literature \cite{corti_discontinuous_2023,corti_positivity_preserving}. The preceding investigation applied the FK equation and despite minor discrepancies arising from the distinct diffusion tensor employed, our analysis reveals a prevailing alignment in the propagation directions of the misfolded protein. Starting from the motor nucleus the misfolded protein spread towards the lower brain stem, subsequently ascending to involve the cerebral cortex in the advanced stages of the pathological process. This observed activation pattern within distinct brain regions aligns with the theoretical tenets elucidated in the Braak staging theory \cite{braakStagingBrainPathology2003a}.
\section{Conclusions}
\label{conclusion}
In this article, we have introduced a Polygonal Discontinuous Galerkin method as a computational approach for simulating the heterodimer model. Our work includes demonstrating the stability of the numerical solution and establishing an error estimate for the semi-discrete solution. We performed a convergence analysis with respect to mesh refinements and the order of approximation $p$. Finally, we have numerically verified the theoretical convergence rate with respect to time discretization methods. Then, we have accessed the ability of the method to accurately reproduce a travelling-wave solution, typical of this type model, unless not guaranteed for all the parametrical choices. Our findings have affirmed the method's ability to accurately reconstruct the propagation front of the solution, particularly when higher-order methods are employed, even on coarser grids. 
\par
In the end, we have conducted simulations considering the application of $\alpha$-synuclein spreading across the brain section, in Parkinson's disease. Our results have been validated by comparing them with simulations based on the Fisher-Kolmogorov equation and clinical data from the literature. In conclusion, the heterodimer model has demonstrated its utility as a valuable tool for advancing research in the mathematical modelling of neurodegenerative diseases.
\par
A logical progression of our research involves extending its application to a three-dimensional representation of the brain. Furthermore, an additional prospective direction involves conducting a comprehensive quantitative analysis of the influence exerted by all the parameters within the system on the solution. The heterodimer model can also find application in the study of other prion-like proteins, such as $\beta$-amyloid and tau proteins, which are pivotal in Alzheimer's disease. The extension of its applicability enriches our comprehension of neurodegenerative conditions and contributes to the progression toward more precise and less invasive patient-specific healthcare approaches.

\section*{Acknowledgments}
The brain MRI images were provided by OASIS-3: Longitudinal Multimodal Neuroimaging: Principal Investigators: T. Benzinger, D. Marcus, J. Morris; NIH P30 AG066444, P50 AG00561, P30 NS09857781, P01 AG026276, P01 AG003991, R01 AG043434, UL1 TR000448, R01 EB009352. AV-45 doses were provided by Avid Radiopharmaceuticals, a wholly-owned subsidiary of Eli Lilly.

\section*{Declaration of competing interests}
The authors declare that they have no known competing financial interests or personal relationships that could have appeared to influence the work reported in this article.

\bibliographystyle{hieeetr}
\bibliography{sample.bib}
\end{document}